\documentclass{article}

\usepackage[T1]{fontenc}
\usepackage[utf8]{inputenc}

\usepackage{amsmath}
\usepackage{amssymb}
\usepackage{graphicx, float}
\usepackage{mathtools}
\usepackage{xcolor}

\usepackage{amsthm}
\usepackage{titlesec}
\usepackage{thmtools}
\usepackage{enumitem}

\usepackage{
    nameref,
    hyperref,
    cleveref
}

\hypersetup{
    unicode=true,
    pdftitle={Point-shifts of Point Processes on Topological Groups},
    pdfauthor={James T. Murphy III},
    colorlinks,
    linktocpage,
    citecolor=blue,
    linkcolor=red,
    urlcolor=blue,
    filecolor=green
}

\titlelabel{\thetitle.\quad}

\renewcommand{\leq}{\leqslant}
\renewcommand{\geq}{\geqslant}

\newcommand{\pmat}[1]{\begin{pmatrix}#1\end{pmatrix}}

\newtheorem{theorem}{Theorem}[section]
\newtheorem{proposition}[theorem]{Proposition}
\newtheorem{lemma}[theorem]{Lemma}
\newtheorem{corollary}[theorem]{Corollary}
\newtheorem{conjecture}[theorem]{Conjecture}
\theoremstyle{definition}
\newtheorem{definition}[theorem]{Definition}
\newtheorem{example}[theorem]{Example}

\newcommand{\R}{\mathbb{R}}
\newcommand{\F}{\mathcal{F}}
\newcommand{\I}{\mathcal{I}}
\newcommand{\B}{\mathcal{B}}
\newcommand{\E}{\mathbf{E}}
\renewcommand{\C}{\mathcal{C}}
\renewcommand{\P}{\mathbf{P}}
\newcommand{\X}{\mathbb{X}}
\newcommand{\A}{\mathcal{A}}
\newcommand{\Z}{\mathbb{Z}}
\newcommand{\N}{\mathbb{N}}

\renewcommand{\H}{\mathfrak{H}}
\renewcommand{\L}{\mathcal{L}}

\newcommand{\MM}{\mathbf{M}}

\newcommand{\m}{\mathfrak{m}}
\newcommand{\n}{\mathfrak{n}}
\newcommand{\h}{\mathfrak{h}}
\newcommand{\V}{V}
\newcommand{\EE}{E}
\renewcommand{\v}{v}
\newcommand{\e}{g}
\renewcommand{\o}{o}
\renewcommand{\G}{G}
\newcommand{\GG}{\boldsymbol{\G}}
\newcommand{\oo}{\boldsymbol{\o}}

\newcommand{\calG}{\mathcal{G}}
\newcommand{\calGs}{\mathcal{G}_{*}}
\newcommand{\calGss}{\mathcal{G}_{**}}
\renewcommand{\t}{\mathfrak{t}}

\let\intensity\gamma
\let\Haar\lambda
\let\embedding\eta
\renewcommand{\gamma}{\textcolor{red}{\tt{Use \backslash intensity}}}
\renewcommand{\lambda}{\textcolor{red}{\tt{Use \backslash Haar}}}
\renewcommand{\eta}{\textcolor{red}{\tt{Use \backslash embedding}}}
\newcommand{\textdefn}{\textbf}
\newcommand{\zero}{\boldsymbol{0}}
\DeclareMathOperator{\card}{card}

\begin{document}

\title{Point-shifts of Point Processes on Topological Groups}
\author{James~T.~Murphy~III
    \thanks{James~T.~Murphy~III, The University of Texas at Austin, jamesmurphy@math.utexas.edu}
}
\date{}
\maketitle

\begin{abstract}
    This paper focuses on flow-adapted point-shifts of point processes on topological groups, which
    map points of a point process to other points of the point process in a translation invariant way.
    Foliations and connected components generated by point-shifts are studied, and the cardinality classification
    of connected components, previously known on Euclidean space, is generalized to unimodular groups.
    An explicit counterexample is also given on a non-unimodular group.
    Isomodularity of a point-shift is defined and identified as a key component in
    generalizations of Mecke's invariance theorem in the unimodular and non-unimodular cases.
    Isomodularity is also the deciding factor of when the reciprocal and 
    reverse of a point-map corresponding to a bijective point-shift are equal in distribution.
    Next, sufficient conditions for separating points of a point process are given.
    Finally, connections between point-shifts of point processes and vertex-shifts of unimodular networks
    are given that allude to a deeper connection between the theories.
\end{abstract}



\section*{Acknowledgments}
This work was supported by a grant of the Simons Foundation (\#197982 to The University of Texas at Austin).

\section{Introduction}

\subsection{Background}
Stationary point processes are models of discrete subsets of $\R^d$ that exhibit
``statistical homogeneity''.
That is, their distributions are translation invariant.
Such processes have a wide array of applications in population models, wireless networks,
astrophysical models, or, more abstractly, as the set of vertices of a random graph.
The use of stationarity is a way to encode the idea that no point in the point process is ``special''.
For models of physical systems this often suffices, but $\R^d$ is not always the most natural
space on which to consider these models.
A population model may be considered on $S^2$ since the Earth is approximately spherical.
A computer network may be considered on a hyperbolic group since some studies suggest that distances on the internet
are hyperbolic~\cite{shavitt2008hyperbolic}.
An astrophysical model concerned with structures on very large scales may be considered
on $\R^3$, on $S^3$, or on a $3$-dimensional hyperbolic group, depending on the yet unknown spatial curvature of the universe.
For these situations on a group $\X$, the concept of $\X$-stationarity, i.e.\ distributional
invariance with respect to the action of $\X$ on a space, for a point process may be used.
The general theory of point processes and random measures has been developed on any locally compact second-countable Hausdorff (LCSH) space $\X$,
and $\X$-stationarity has been studied thoroughly when $\X$ is a LCSH topological group acting on a homogeneous space $S$, cf.~\cite{last2010stationary}.
The theory of $\X$-stationarity on homogeneous spaces is very general,
but in many cases an $\X$-stationary point process on a homogeneous space $S$
may be pushed forward in a natural way to an $\X$-stationary point process on $\X$ itself, and this is the setting
that is assumed here.

This paper will make use of point processes on a general LCSH group $\X$, which is fixed for the remainder of the document.
Denote the Borel sets of $\X$ by $\B(\X)$.
For sake of completeness, the following definition is included.
A \textdefn{point process} on $\X$ is a random element $\m$ in the space $\MM$
of all locally finite counting measures on $\X$, where $\MM$ is endowed with
the cylindrical $\sigma$-algebra generated by the mappings $\mu \in \MM \mapsto \mu(B)$ for each $B\in \B(\X)$.
All point processes $\m$ in this paper are assumed to be simple, meaning every atom of $\m$ has mass $1$.
In this case, $\m$ can and will be identified with its support, which is a discrete subset of $\X$.
Also, when both arguments need to be specified, the notation $\m(\omega,B)$ will be used instead of $\m(\omega)(B)$.

The following framework for dealing with $\X$-stationary point processes was developed in~\cite{last2010stationary,last2008modern}.
A \textdefn{stationary framework} $(\Omega, \A, \theta,\P)$ on $\X$ is a probability space 
$(\Omega,\A,\P)$ equipped with a measurable and $\P$-invariant left $\X$-action $\theta:\X \times \Omega \to \Omega$,
called a \textdefn{flow},
which will be identified with the family of mappings $\{\theta_x\}_{x \in \X}$ defined by $\theta_x\omega := \theta(x,\omega)$ for $x\in \X, \omega \in \Omega$.
A point process $\m$ is \textdefn{flow-adapted}
if 
\begin{equation}\label{eqtn:covariance}
    \m(\theta_x\omega ,B) = \m(\omega, x^{-1}B),\qquad \forall x\in\X,\omega\in\Omega,B\in\B(\X).
\end{equation}
Another way of expressing \eqref{eqtn:covariance} is
\[
    \m\circ \theta_x = T_x\m ,\qquad \forall x\in\X,
\]
where for $\mu \in \MM$ and $x \in \X$ the translated measure $T_x\mu$ is defined by $T_x\mu(B) := \mu(x^{-1}B)$ for all $B \in \B(\X)$.
Under these assumptions, any flow-adapted $\m$ is $\X$-stationary in the usual sense that $\m$ and $T_x\m$
have the same distribution for all $x\in\X$.
For the remainder of the document fix a stationary framework $(\Omega,\A,\theta,\P)$ on $\X$.
All point processes introduced in this document are assumed to be flow-adapted.

Many models are concerned with more than just the statistical properties of locations of points though.
Often dynamics on the points are of primary interest.
For instance, in a wireless network, one may be interested in a universally agreed upon protocol for determining an optimal route for packet transmission
that depends only on local information.
At a certain instant of time, such a protocol would determine a point-shift.
A \textdefn{point-shift} on a point process $\m$ is a measurable 
map $\H:\Omega \times \X\to \X$ on the
support of $\m$, i.e.\ for $\P$-a.e.\ $\omega \in \Omega$,
\begin{equation}\label{eqtn:point-shift-defn}
    \H(\omega,X) \in \m(\omega),\qquad \forall X \in \m(\omega).
\end{equation}
From now on, all point-shifts considered in this document are assumed to be \textdefn{flow-adapted}
in the sense that
\begin{equation}\label{eqtn:point-shift-covariance}
    \H(\theta_y\omega,yx) =y\H(\omega,x), \qquad \forall x,y \in \X,\omega \in \Omega.
\end{equation}
If unspecified, $\H(\omega,x):=x$ for $x \notin \m(\omega)$.
Dependence on $\omega$ is usually dropped and $\H(X)$ is written instead of
$\H(\omega,X)$.
Say that $\H$ has a functional property, e.g.\ bijectivity, injectivity,
surjectivity, if for $\P$-a.e.\ $\omega\in\Omega$, $\H(\omega,\cdot)$ has
the property on the support of $\m(\omega)$.
Some references refer to point-shifts as point allocations as well.

Point-shifts of point processes on $\R^d$ and the random graphs they generate 
(by considering each point a vertex and a directed edge from each point to its image under the point-shift) have been well studied,
cf.~\cite{baccelli2016foliations,baccelli2017point}.
Point-shifts of a point process on a general LCSH group have not yet been studied, and 
thus the aforementioned models, e.g.\ wireless networks on hyperbolic groups, in which
dynamics of points are of central importance, were not practical due to lack of theoretical machinery.
The main focus of this document is to extend many known results of point-shifts
of stationary point processes on $\R^d$ to point-shifts of $\X$-stationary point processes on $\X$.
In many ways this task is straightforward, but some proofs on $\R^d$ made specific use of the structural
properties of $\R$, in particular that there is a translation-invariant order on $\R$.
New proofs using mass transport techniques are given instead.
The primary tools that will be used are Palm probabilities and the mass transport principle.

The definition of Palm probability measures, as given in~\cite{last2008modern} and applied to the current setting for point processes, follows.
Fix, for the remainder of the document, a left-invariant Haar measure $\Haar$ on $\X$.
Also, for the remainder of the section, suppose $\m$ is a flow-adapted point processes with finite and nonzero intensity,
that is $\Lambda(\cdot) := \E[\m(\cdot)]$ is locally finite and not the zero measure.
The \textdefn{Palm probability measure} of $\m$, denoted $\P^\m$, is defined by,
\[
    \P^\m(A) := \frac{1}{\intensity}\E\int_\X 1_{\theta_x^{-1} \in A} w(x)\,\m(dx), \qquad \forall A \in \A,
\]
where $\intensity = \frac{\E[\m(B)]}{\Haar(B)}$ for any $B\in \B(\X)$ with $\Haar(B) \in (0,\infty)$, and
$w:\X\to\R_+$ is any non-negative measurable function with $\int_\X w\,d\Haar= 1$.
Note that $\intensity\in (0,\infty)$ is uniquely determined and $\P^\m$ is independent of the choice of $w$.
Expectation with respect to $\P^\m$ is denoted $\E^\m$.
The Palm probability measure $\P^\m$ makes rigorous what is meant by the view of the world from a typical point's perspective.
Heuristically, it is the reference probability measure conditioned on the event that $e \in \m$, where $e$ is the neutral element of $\X$.

It is also possible to convert between $\P$-a.s.\ and $\P^\m$-a.s.\ events in the following manner.
Intuitively, that which happens almost surely from the typical point's perspective happens almost surely from every point's perspective simultaneously,
and vice-versa.

\begin{theorem}\label{palmpalmostsureequivalence}
    Let $A \in \A$. Then the following are equivalent:
    \begin{enumerate}[label=(\alph*)]
        \item $\P^\m(A) = 1$,
        \item $\P(\m(x\in \X: \theta_x^{-1} \notin A) =0)=1$,
        \item $\P^\m(\m(x\in \X: \theta_x^{-1} \notin A) =0)=1$.
    \end{enumerate}
\end{theorem}

A proof of~\Cref{palmpalmostsureequivalence} could not be found in the literature, so one
is given in the appendix.
\Cref{palmpalmostsureequivalence} may be used
to translate between definitions under $\P^\m$ and definitions under $\P$.
For example, a \textdefn{point-map} on $\m$ is a measurable map $\h: \Omega \to \X$ such that $\h(\omega) \in
\m(\omega)$ for $\P^\m$-a.e.\ $\omega \in \Omega$.
There is a natural correspondence between point-shifts and point-maps.
Namely, if $\H$ is a point-shift, then $\h(\omega):=\H(\omega,e)$ is a point-map, and if
$\h$ is a point-map, then $\H(\omega,X):=X\h(\theta_X^{-1}\omega)$ is a
point-shift, and these operations are inverses.
\Cref{palmpalmostsureequivalence} ensures that changing the definition of $\H$ on a $\P$-null set
will only change the corresponding $\h$ on a $\P^\m$-null set, and vice-versa.
The unfamiliar reader may see \Cref{convertPandPalmDefn} in the appendix for the details of how to 
convert definitions under $\P$ and $\P^\m$ more generally.

Since point-shifts can be seen as special cases of mass transport kernels,
the mass transport principle for $\X$-stationary point processes is also a crucial tool for this study.
In what follows and for the rest of the document, $\Delta:\X\to(0,\infty)$ is the modular function of $\X$, i.e.\ 
$\Haar(Bx) = \Delta(x)\Haar(B)$ for all $B \in \B(\X)$.
Applied to the current setting, the mass transport principle for $\X$-stationary point processes
takes the following form.

\begin{theorem}[Mass Transport Theorem]\cite{last2008modern}
    For all \textdefn{diagonally invariant} $\tau$, i.e.\
    measurable $\tau:\Omega \times \X \times \X \to \R_+$ 
    invariant in the sense that
    \begin{equation}\label{transportfunctioneq}
        \tau(\theta_z \omega, zx,zy) = \tau(\omega,x,y) =: \tau(x,y),\qquad
        \omega\in\Omega,\quad x,y,z \in \X,
    \end{equation}
    it holds that
    \begin{equation}\label{mtp0}
        \E^\m \int_\X \tau(e,y)\,\m(dy) 
        = \E^{\m}\int_\X \tau(x,e)\Delta(x^{-1})\,\m(dx).
    \end{equation}
\end{theorem}

Interpret $\tau(\omega,x,y)$ as the amount of mass sent from $x$ to $y$ on
the outcome $\omega$.
Under $\E^{\m}$, $e$ is a point of $\m$.
Thus the left side of \eqref{mtp0} is an average of mass sent out of $e \in
\m$ to all points $\m$.
On the other hand, the right side of \eqref{mtp0} is a weighted average of mass received by $e
\in \m$ from all points of $\m$.
If $\Delta(x) = 1$ for all $x \in \X$, i.e.\ if $\X$ is unimodular, then the mass
transport formula is the one expected from the case of translations on
$\R^d$, which says that the average mass a typical point of $\m$ receives equals the average mass a typical point of $\m$ sends.

\subsection{Results}

Since the modular function $\Delta$ of $\X$ appears in the mass transport principle,
unimodularity of $\X$ turns out to be an influential property
for whether certain results about point-shifts of point processes on $\R^d$ generalize.
Moreover, the role that $\Delta$ plays in the story of point-shifts is crucial to understanding
why exactly some results generalize and others do not.
In \Cref{sec:cardinality-classification,sec:classification-counterexample}, a stark contrast between
the behavior of point-shifts in unimodular and non-unimodular cases is shown.
However,  there is more subtlety in the behavior of point-shifts when one does not assume that $\Delta(x) = 1$ for all $x \in \X$, but 
rather only that a point-shift is, in a natural sense, compatible with the modular function $\Delta$.
This motivates the primary new definition of this research.
\begin{definition}\label{defn:isomodular}
    A point-shift $\H$ on a point process $\m$ is \textdefn{isomodular} if $\P$-a.s.\ one has $\Delta(X) = \Delta(\H(X))$
    for all $X \in \m$.
\end{definition}
Intuitively, although $\X$ may not be unimodular, an isomodular point-shift 
acts on slices of $\X$ of the same modularity so that the mass transport principle works
the same as it would on a unimodular space.
\Cref{sec:mecke-invariance-thm,sec:reciprocal-and-reverse,sec:separating-points}
outline some of the subtleties of isomodularity.

In \Cref{sec:cardinality-classification}, the structure of the components of the random graph generated by a point-shift is studied when $\X$ is unimodular.
For applications, it is interesting to study the behavior of points under repeated applications of a point-shift.
How many points are in a given graph component?
Given a point, how many other points will merge with it under repeated application of a point-shift?
For $\X=\R^d$, the cardinality classification theorem proved in~\cite{baccelli2016foliations}
asserts that answers to these questions fundamentally place every component $C$ into one of three types: $\F/\F, \I/\F$, 
or $\I/\I$ with $\I$ (infinite) and $\F$ (finite) representing the answers to the two questions respectively.
The type of a component $C$ determines whether it is acyclic, the number of bi-infinite paths it contains,
whether it can be linearly ordered in a flow-adapted way, and whether any points remain after an infinite number
of applications of the point-shift.
\Cref{cardinalityclassification} shows that the classification when $\X=\R^d$ extends verbatim to $\X$ 
when $\X$ is unimodular.

\Cref{sec:classification-counterexample} shows that unimodularity of $\X$ is crucial
to the cardinality classification theorem. 
In this section, $\X$ is chosen to be an explicit non-unimodular group, the $ax+b$ group, and a point-shift $\H$ 
is given for which the previous cardinality classification fails in many respects.

It is a classical result on  $\R^d$, cf.~\cite{heveling2005characterization}, that a point-shift preserves the Palm probability
measure of a point process if and only if the point-shift is almost surely bijective on the support of the process.
This result is commonly referred to as Mecke's invariance theorem or, in some cases, Mecke's point-stationarity theorem.
In \Cref{sec:mecke-invariance-thm} it is shown that Mecke's invariance theorem holds if $\X$ is unimodular.
Moreover, when $\X$ is not necessarily unimodular, the class of bijective point-shifts that preserve Palm probabilities
is identified as the bijective isomodular point-shifts, the point-shifts that preserve the modular function of the group.
Mecke's invariance theorem in the unimodular case is \Cref{meckethm}, and the identification of isomodular
point-shifts as being the ones that preserve Palm probabilities (which may be considered a generalization of Mecke's invariance theorem for the non-unimodular case) 
is \Cref{isomodularinvariance}.

\Cref{sec:reciprocal-and-reverse} continues with the study of isomodularity and investigates for bijective point-shifts the distributional relationship
between the reciprocal of the corresponding point-map and the reverse point-map,
which corresponds to running the point-shift backwards in time.
In particular, \Cref{hinversehminussamelaw} shows that, amongst bijective point shifts,
the isomodular point-shifts are exactly the ones for which the reciprocal of the point-map
and the reverse point-map are equal in distribution.
Note, however, that in most cases the reciprocal of a point-map on $\m$ is itself not a point-map on $\m$.

\Cref{sec:separating-points} studies different ways in which functions separate points of a point process.
For example, given a function $f:\X \to S$ for some set $S$ and a point process $\m$, when are the values of $f(X)$ distinct for all $X \in \m$?
\Cref{conditionforseparatingpoints} gives some sufficient conditions on $f$ for separating points of the point process.
This is useful to show that some point-shifts are well-defined when specifying where to send a point
by comparing values of $f(X)$ for different $X \in \m$.

Finally, many of the results of the previous sections have analogs to
known results from another
framework for dealing with random graphs and dynamics on their vertices.
\Cref{sec:connections-unimodular-networks} studies this other framework, whose objects of study are called
random (rooted) networks and where dynamics on the vertices of these networks are called vertex-shifts.
Random networks are technical objects that will be introduced properly in \Cref{sec:connections-unimodular-networks}.
Heuristically, they model random graphs where a vertex has been singled out and designated the root,
and where vertices and edges may be endowed with extra associated information called marks.
An analogous requirement to $\X$-stationarity for random networks is called unimodularity.
In a unimodular random network, the root is picked ``uniformly'' from the vertex set
in the sense that unimodular networks are defined as those that satisfy a certain mass transport principle
for mass into and out of the root.
This mass transport principle is similar to the mass transport theorem for point processes if the neutral element
$e \in\X$ under a Palm probability measure were considered the ``root'' of the point process.
Because unimodular networks satisfy this mass transport principle, they exhibit analogs of many of the theorems
of point processes such as Mecke's invariance theorem and the cardinality classification of components of point-shifts~\cite{baccelli2016networks}.
The parallels between point processes and random networks
suggest the following motivating questions of the present research.
\begin{enumerate}
    \item{} Given an $\X$-stationary point process, when can the Palm version of it be seen as an embedding of a unimodular network?
    \item{} Given a unimodular network, when is it possible to find an $\X$-stationary point process
        such that the Palm version of the point process is an embedding of the given unimodular network?
\end{enumerate}
Some progress in the answering the first question is made in Section~\ref{sec:connections-unimodular-networks}, 
where the problem is reduced to an invariant geometry problem on the underlying space,
which is conjectured to always be solvable when $\X$ is unimodular.
The results of previous sections also indicate that when $\X$ is not unimodular, one should not expect either question
to be answered affirmatively.

\section{Point-shift basics and notation}

In this section, some notation and results that are used throughout the rest of the text are collected.
Fix a point-shift $\H$ with corresponding point-map $\h$ on a flow-adapted point process $\m$ with intensity $\intensity \in (0,\infty)$
for the remainder of the section.
In order to better suit the random graph setup desired in applications,
define the functions (omitting $\omega$ dependence):
\begin{definition}
\begin{itemize}
    \item[]
    \item{}
        \textdefn{Edge indicator}: $\tau^\H(x,y) := 1_{x,y\in\m,\H(x)=y}$
        for all $x,y \in \X$.
    \item{}
        \textdefn{Out-neighbors} and \textdefn{in-neighbors} of $e$ under $\P^\m$:
        \begin{align*}
            h^+ &:= \{Y \in \m: \tau^\H(e,Y)=1\} = \{\h\},\\
            h^- &:= \{X \in \m: \tau^\H(X,e)=1\} 
            = \{Y \in \m: \h(\theta_Y^{-1}\omega) = Y^{-1}\}.
        \end{align*}
    \item{}
        \textdefn{Out-neighbors} and \textdefn{in-neighbors} under $\P$ or $\P^\m$:
        \begin{align*}
            H^+(X) &:= 
            X h^+(\theta_X^{-1})
            =\{Y \in \m: \tau^\H(X,Y) = 1\} 
            = \{\H(X)\}, \\
            H^-(X) &:= 
            X h^-(\theta_X^{-1})
            =\{Y \in \m: \tau^\H(Y,X) = 1\}
            = \{Y \in \m:\H(Y)=X\}
        \end{align*}
        for all $X \in \m$.
    \item{} \textdefn{Preimage} of $e$ under $\P^\m$: if $\P^\m$-a.s.\ $\card(h^-)=1$,
        then $\h^-$ is defined to be the unique element in $h^-$.
        By \Cref{palmpalmostsureequivalence}, this is equivalent to $\H$ being bijective.
    \item{} \textdefn{Reverse point-shift} under $\P$: if $\P$-a.s.\ $\card(H^-(X))=1$ for all $X
        \in \m$ (equivalently $\P^\m$-a.s.\ $\card(h^-)=1$), then $\H^-(X)$ is defined to be the unique element in
        $H^-(X)$.
        Note that $\H^-$
        is defined if and only if $\H$ is bijective.
        In this case $\P$-a.s.\ $\H(\H^-(X)) =\H^-(\H(X)) = X$ for all $X \in \m$.
        That is, $\H$ and $\H^-$ are inverses on the support of $\m$.
\end{itemize}
\end{definition}
With these definitions, $\tau^\H$ is diagonally invariant,
and the definitions of $h^+$, $h^-$, $\h^-$ under $\P^\m$ are equivalent to the definitions 
of $H^+$, $H^-$, $\H^-$ under $\P$ or $\P^\m$ via \Cref{palmpalmostsureequivalence}.

With the mass transport theorem and \Cref{palmpalmostsureequivalence}, the following may be obtained
in a straightforward manner.

\begin{proposition}\label{structuralpropertiesofpointshifts}
    The following hold:
    \begin{enumerate}[label=(\alph*)]
        \item $\P$-a.s.\ every $X \in \m$ is the image under $\H$ of at least (resp.\ at most) $k$
            distinct points of $\m$ if and only if $\P^\m$-a.s.\ $\card(h^-) \geq k$ (resp. $\leq k$),
        \item $\P$-a.s.\ every $X \in \m$ is the image under $\H$ of finitely (resp.\ infinitely) many
            distinct points of $\m$ if and only if $\P^\m$-a.s.\ $\card(h^-) < \infty$ (resp.\ $=\infty$),
        \item $\P$-a.s.\ $\H$ is bijective (resp.\ surjective, injective) if and only if
            $\P^\m$-a.s.\ $\card(h^-)=1$ (resp.\ $\geq 1, \leq 1$).
            In particular $\h^-$ and $\H^-$ are well-defined if and only if $\H$ is bijective,
        \item 
        for all $f:\Omega\to \R_+$ measurable, 
        \begin{equation}\label{changeofpalm0}
            \E^\m[f(\theta_\h^{-1})\Delta(\h^{-1})] = \E^\m[f\card(h^-)].
        \end{equation}
        In particular, the following mass flow relationship for point-shifts holds
        \begin{equation}\label{changeofpalm1}
            \E^\m[\Delta(\h^{-1})] = \E^\m[\card(h^-)],
        \end{equation}
        \item $\P$-a.s.\ every $X \in \m$ is the image under $\H$ of at least (resp.\ at most) $k$
            points of $\m$ if and only if for all $f:\Omega \to \R_+$ measurable 
            \[
                \E^\m[f(\theta_\h^{-1})\Delta(\h^{-1})] \geq k \E^\m[f] \qquad (resp.\ \leq k\E^\m[f]),
            \]
        \item (Test for Bijectivity)\footnote{G.\ Last
                also proves this and similar results, e.g.\ Corollary 10.1 in \cite{last2008modern}.}
            $\H$ is bijective if and only if
            for all $f:\Omega \to \R_+$ measurable
            \begin{equation}\label{bijectivitytest}
                \E^\m[f(\theta_\h^{-1})]\Delta(\h^{-1})] = \E^\m[f],
            \end{equation}
        \item If $\H$ is bijective, also
            \begin{equation}\label{changeofpalm}
                \E^\m\left[f(\theta_\h^{-1})\right] = \E^\m\left[\frac{f}{\Delta(\h^-)}\right],
            \end{equation}
        \item If $\P$-a.s.\ every $X \in \m$ is the image under $\H$ of at least
            (resp.\ at most) $k$ points of $\m$, then $\E^\m[\Delta(\h^{-1})] \geq k$ (resp.\ $\leq k$),
        \item If $\E^\m[\Delta(\h^{-1})] < \infty$,
            every $X \in \m$ is the image of only finitely many $Y \in \m$
            under $\H$,
        \item\label{test} If $\E^\m[\Delta(\h^{-1})]=1$, then $\H$ is injective if and only if it is surjective.
            In particular, this is automatic if $\X$ is unimodular.
   \end{enumerate}
\end{proposition}

\begin{proof}
    \begin{enumerate}
        \item[]
        \item[] \emph{(a),(b),(c):} Direct application of \Cref{palmpalmostsureequivalence}.
        \item[] \emph{(d):} Apply the mass transport theorem with the diagonally invariant function
            \[
                \tau(\omega,x,y) := f(\theta_y^{-1}\omega)1_{x,y\in\m(\omega), y=\H(x)}\Delta(y^{-1}x).
            \]
        \item[] \emph{(e):} Apply \emph{(a)} and \emph{(d)}.
        \item[] \emph{(f):} Apply \emph{(e)} with $k:=1$.
        \item[] \emph{(g):} Replace $f$ with $\frac{f}{\Delta(\h^-)}$ in \emph{(d)}
            and use the fact that $\P^\m$-a.s.\ 
                \[
                    \h^-(\theta_\h^{-1}) 
                    =\h^{-1}(\h \h^-(\theta_\h^{-1}))
                    =\h^{-1}\H^-(\h)
                    =\h^{-1}\H^-(\H(e))
                    = \h^{-1}.
                \]
        \item[] \emph{(h),(i):} Take $f:=1$ in \emph{(d)} and apply \emph{(a)} or \emph{(b)}.
        \item[] \emph{(j):} Take $f:=1$ in \emph{(d)}. Use \emph{(a)}, \emph{(c)}, and the fact that
            a random variable bounded above (or below) by 1 with expectation $1$
            must be constant 1 a.s.
    \end{enumerate} 
\end{proof}

\section{Point-shift Foliations}

\subsection{The Cardinality Classification of Components}\label{sec:cardinality-classification}
In this section the cardinality classification components of point-shifts
in \cite{baccelli2016foliations} is extended to the general stationary framework for
unimodular $\X$.
The classification theorem is \Cref{cardinalityclassification}, and the fundamental
result used in its proof, which says it is impossible to pick out finite
subsets of infinite sets in a flow-adapted manner, is
\Cref{noflowadaptedselection}.

\textbf{Throughout this section, $\X$ is assumed to be unimodular.}
Fix for the rest of the section a flow-adapted simple point process $\m$ on $\X$ with intensity $\intensity \in (0,\infty)$,
and a point-map $\h$ on $\m$ with 
corresponding point-shift $\H$.
The wording of proofs is substantially cut down by thinking of $\H(X)$
as the \textdefn{father} of $X$.
For example, the \textdefn{children} of $X$ are the
$Y \in \m$ such that $\H(Y) = X$.
Next appear the necessary ingredients needed for the
classification theorem. 

\begin{definition}
    The iterates $\H^n$ are defined by repeatedly applying the point-shift
    $\H$. That is, $\H^0(X) :=X$ and $\H^{n+1}(X) := \H(\H^n(X))$ for all $X \in
    \m$.
    Elements $Y\in \m$ that are in the image $\H^n(\m)$ for all $n \in \N$ are
    called \textdefn{primeval}, and $\H^\infty(\m)$ will denote the set of all
    primeval elements of $\m$.
    Here $\H^n(\m)$ is considered as a set, i.e.\ multiplicities are ignored, for
    all $n \leq \infty$.
    Moreover, $\H^n(\m)$ is a flow-adapted simple point process for any $n \leq \infty$.
\end{definition}

Random graphs will be used throughout this section.
Here a random (directed) graph $G$ on $\X$ is specified with a random variable $N$ taking values in $\N \cup \{\infty\}$ 
and
random elements $\{x_i\}_{i \in \N}$ in $\X$ with $V(G) := \{ x_i : i \leq N\}$, and measurable indicators $\{\xi_{ij}\}_{i,j \in \N}$ 
with $E(G):= \{ (x_i,x_j) : i,j \leq N, \xi_{ij}=1\}$.
A random subset $C$ of vertices of $G$ is a map on $\Omega$ taking values in the subsets of $V(G)$ such that $1_{x_i \in C}$ is measurable for each $i$.
Similarly, a random (countable) collection $\C = \{C_i\}_{1 \leq i \leq N^\C}$ of subsets of vertices of $G$
is identified with a random variable $N^\C$ taking values in $\N \cup \{\infty\}$ and random subsets $\{C_i\}_{i \in \N}$
with the elements of $\C$ being defined as $\{C_i : i \leq N^\C\}$.
In all cases of interest for the present study, the specific numbering of vertices in a random graph or elements
of a random collection are of no interest and will not be given upon defining the graph or collection.

The adjective flow-adapted has already been defined for point processes and point-shifts.
The same adjective will also be used for random graphs and for random collections.

\begin{definition}
A random graph $G$ on $\X$ is \textdefn{flow-adapted} if for all $\omega\in \Omega$ and all $X,Y,z \in \X$
and one has $X \in V(G(\omega))$ if and only if $zX \in V(G(\theta_z\omega))$ and $(X,Y) \in E(G(\omega))$
if and only if $(zX,zY) \in E(G(\theta_z\omega))$.
A random collection $\C = \{C_i\}_{1\leq i \leq N}$ 
is \textdefn{flow-adapted} if for all $\omega \in \Omega,z \in \X$,
one has $N(\theta_z \omega) = N(\omega)$ and
there is a permutation $\pi(\omega)$ of $1,\ldots, N(\omega)$ such that
$C_i(\theta_z\omega) = \{z x: x \in C_{\pi(\omega)}(\omega)\}$ for each $i \leq N(\omega)$.
That is, $\C(\theta_z\omega)$ contains the same elements as $\C(\omega)$, shifted by $z$, and possibly enumerated in a different order.
\end{definition}

Now the random graph generated by the point-shift $\H$ is defined.

\begin{definition}
    The random graph $G^\H$ is defined to have vertices at the points of $\m$ and directed edges from each $X \in \m$
    to $\H(X)$.
\end{definition}

Two natural equivalence relations on the vertices of $G^\H$ are defined by connected components and foils.

\begin{definition}
    The set of undirected connected components of $G^\H$ is denoted by $\C^\H$ and the component
    of $X \in \m$ is denoted $C^\H(X)$.
    Then $X, Y \in \m$ are in the same component if and only if there are $n,m \in \N$
    such that $\H^m(X) = \H^n(Y)$.
    That is, $C^\H(X)$ is the set of all \textdefn{relatives} of $X$.
    The graph $G^\H$ is flow-adapted, and hence so is $\C^\H$.
\end{definition}

\begin{definition}
    The \textdefn{foliation} $\L^\H$ is defined to be the set of \textdefn{foils}
    $L^\H(X)$ of
    $\H$ for $X \in \m$,
    which are equivalence classes under the equivalence relation where $X,Y\in \m$
    are equivalent if and only if there is $n \in \N$ such that $\H^n(X) = \H^n(Y)$.
    That is, $L^\H(X)$ is the relatives of $X$ from the same \textdefn{generation}
    as $X$.
    The foliation $\L^\H$ is flow-adapted, and $\L^\H$ is a subdivision of $\C^\H$.
    For a foil $L$, also denote $L_+ := L^\H(\H(X))$ for any $X \in L$.
    Note that if $X,X' \in L$ then $L^\H(\H(X)) = L^\H(\H(X'))$ so $L_+$ is
    well-defined.
    If there is $Y \in \m$ such
    that $\H(Y) \in L$, then set $L_- := L^\H(Y)$.
    Then $L_-$ is well-defined because if $Y,Y'$ are both such that
    $\H(Y),\H(Y') \in L$, then $L(Y) = L(Y')$.
    It holds that $(L_+)_- = L$ and when $L_-$ exists $(L_-)_+ = L$.
\end{definition}

It will be important later to know that the graph $G^\H$ is locally
finite.
The following result, generalizing one in \cite{baccelli2016foliations}, guarantees this.
It crucially relies on the unimodularity
of $\X$.

\begin{proposition}\label{finitelymanychildren}
    Let $D_n(X)$ denote the $n$-th order descendants of $X$, i.e.\ $D_n(X):= \{
    Y \in \m : \H^n(Y) = X \}$.
    Also let $D(X) := \bigcup_{n=1}^\infty D_n(X)$ be all descendants of $X$.
    Then with $d_n(X)$ $:=$ $\card (D_n(X))$, $d(X) := \card (D(X))$,
    one has for every $n \geq 0$ that
    $\E^\m[d_n(e)] = 1$.
    In particular, $d_n(e)$ is $\P^\m$-a.s.\ finite, or equivalently $\P$-a.s.\ 
    every $X \in \m$ has $d_n(X)$ finite.
    If, in addition, $G^\H$ is $\P^\m$-a.s.\ acyclic, then
    $\E^\m[d(e)] = \infty$.
\end{proposition}

\begin{proof}
    $\H^n$ is a point-shift in its own right, so the mass flow relationship~\eqref{changeofpalm1}
    implies $\E^\m[d_n(e)] = \E^\m[\card (D_n(e))] = 1$ since $\X$ is unimodular.
    Thus $d_n(e)<\infty$, $\P^\m$-a.s., and hence $\P$-a.s.\ $d_n(X)<\infty$
    for all $X \in \m$ by \Cref{palmpalmostsureequivalence}.
    Moreover, when $G^\H$ is acyclic, the $D_n$ partition $D$ and hence
    $\E^\m[d(e)] = \sum_{n=1}^\infty \E^\m[d_n(e)] = \infty$.
\end{proof}

The primary tool needed to prove the classification theorem follows.
It says that it is not possible to extract finite subsets of infinite
subsets of $\m$ in a flow-adapted way.
The proof is modified from the argument proving a similar result for 
unimodular networks given by Lemma 3.23 in \cite{baccelli2016networks}.

\begin{theorem}\label{noflowadaptedselection}
    Let $\mathfrak{N} = \{ \mathfrak{N}_i \}_{1 \leq i \leq N}$ be a flow-adapted
    collection of infinite measurable subsets of $\m$ and let $k$ be the number
    of $i$ such that $e \in \mathfrak{N}_i$.
    Suppose that $\E^\m[k] < \infty$.
    If $\n$ is a measurable flow-adapted subset of $\m$ for which $\P$-a.s.\
    $\card(\n \cap \mathfrak{N}_i) < \infty$ for each $i$, then $\P$-a.s.\ $\n
    \cap \mathfrak{N}_i = \emptyset$ for all $i$.
    In particular, if $\n \subseteq \bigcup \mathfrak{N}$, then $\P$-a.s.\
    $\n = \emptyset$.
\end{theorem}

\begin{proof}
    Define
    \[
        \tau(\omega,x,y) :=   \sum_{i=1}^{N(\omega)} 1_{x,y \in
            \mathfrak{N}_i(\omega),y\in\n(\omega)}\frac{1}{\card(\n(\omega)\cap\mathfrak{N}_i(\omega))}.
    \]
    The assumptions about flow-adaptedness of $\mathfrak{N}$, $\n$, and $\m$,
    imply that $\tau$ is diagonally invariant.
    Then $\int_\X \tau(e,y)\,\m(dy) = k$ by construction since $e$ is
    in $k$ of the $\mathfrak{N}_i$.
    Also $\int_\X \tau(x,e)\,\m(dx)=\infty$ if $e \in
    \n \cap \mathfrak{N}_i$ for some $i$ because the $\mathfrak{N}_i$ are
    infinite.
    But the mass transport theorem implies
    \begin{align*}
        \E^\m\int_\X \tau(x,e)\,\m(dx)
        = \E^\m \int_\X \tau(e,y)\,\m(dy)
        =\E^\m[k] 
        < \infty,
    \end{align*}
    and thus it must be that $\P^\m$-a.s.\ $e \notin \n \cap \mathfrak{N}_i$ for
    any $i$.
    Equivalently, $\P$-a.s.\ for all $X \in \m$ it holds that 
    $X \notin \n \cap \mathfrak{N}_i$ for any $i$.
    Since $\n \cap \mathfrak{N}_i \subseteq \m$ for each $i$,
    it follows that $\P$-a.s.\ 
    $\n \cap \mathfrak{N}_i = \emptyset$ for all $i$.
\end{proof}

Note that, by \Cref{palmpalmostsureequivalence}, the condition $\E^\m[k] < \infty$ appearing in \Cref{noflowadaptedselection}
is automatically satisfied if the $\mathfrak{N}_i$ are pairwise disjoint, or more generally
if there is a constant $n$ such that almost surely no $X \in \m$ appears in more than $n$ of the $\mathfrak{N}_i$,
as this would imply $k \leq n$, $\P^\m$-a.s.

More information follows about the
structure of the locally finite graph $G^\H$.
In particular, cycles in components are unique, infinite components are
acyclic, foils in infinite components can be ordered like $\N$ or $\Z$ in a
flow-adapted way, and $\H$ acts bijectively on the primeval elements.

\begin{lemma}\label{infinitecomponentsacyclic}
    $\P$-a.s.\ a connected component $C$ of $G^\H$ is either an infinite tree or has 
    exactly one (directed) cycle $K(C)$ for which for all $Y \in C$ there is $n \in \N$
    such that $\H^n(Y) \in K(C)$.
    Moreover, $\P$-a.s.\ there are no infinite components with a cycle.
\end{lemma}

\begin{proof}
    The fact that all elements in $C$ are connected and have out-degree 
    $1$ implies there can be at most one cycle.
    If there are no cycles then $C$ must be infinite since applying $\H$ to any
    element repeatedly must never repeat an element.
    Otherwise there is one cycle $K(C)$ and connectedness implies for every $Y \in C$
    there is $n \in \N$ with $\H^n(Y) \in K(C)$.
    
    Let $\mathfrak{N}$ be the set of infinite components of $G^\H$ with a cycle, and let
    $\n\subseteq \bigcup \mathfrak{N}$ be the union of all the cycles of these components.
    Since cycles are finite, it follows that $\n \cap C$ is finite for all components $C \in
    \mathfrak{N}$.
    By \Cref{noflowadaptedselection} $\n = \emptyset$ and hence there are no
    infinite components with a cycle $\P$-a.s.
\end{proof}

\begin{definition}
    Within an infinite acyclic connected component $C \in \C^\H$, it is possible to define
    an order, called the \textdefn{foil order}, on the foils $\L^\H(C)$ that are
    subsets of $C$.
    This is accomplished by declaring
    $L^\H(X) < L_+^\H(X)$ for all $X \in C$.
    When thinking of $\H(X)$ as being the father of $X$, the order 
    is that of seniority.
\end{definition}

\begin{lemma}
    The foil order on an infinite acyclic component $C$
    is a total order on $C$ similar to either the order of $\Z$ or $\N$.
\end{lemma}

\begin{proof}
    Fix any $X \in C$.
    Let $L_0 := L^\H(X)$ and recursively define $L_{n+1} := (L_{n})_+$ and if it
    exists $L_{-n-1} :=  (L_{-n})_-$ for $n > 0$.
    Let $L$ be a foil in $C$, then it must be that $L=L_i$ for some $i$.
    Indeed, let $Y \in L$ and by definition of connectedness choose
    $n,m$ such that $\H^n(Y) = \H^m(X) \in L_m$.
    It then follows by induction that $Y \in L_{m-n}$, and hence $L=L^\H(Y) =
    L_{m-n}$.
    Next it is shown that $i \mapsto L_i$ is injective.
    Suppose for contradiction that $L_j= L_{j+N}$.
    Then there are $N$ pairs $(X_i, Y_{i+1})$ with $X_i \in L_i,Y_{i+1} \in
    L_{i+1}$
    such that $\H(X_i) = Y_{i+1}$ for $j \leq i \leq j+N-1$.
    Since $L_j = L_{j+N}$ it follows that $X_j,Y_{j+N} \in L_j$.
    Hence it is possible to choose $n$ such that $\H^n(X_j) = \H^n(Y_{j+N})$ and
    $\H^n(X_i) = \H^n(Y_i)$ for all $j+1 \leq i \leq j+N-1$.
    Assume by induction that for some $k$ one has $\H^{N}(\H^n(X_j)) = \H^{N-k}(\H^n(Y_{j+k}))$.
    Then as long as $k+1 \leq N$,
    \begin{align*}
        \H^N(\H^n(X_j)) 
        &= \H^{N-k}(\H^n(Y_{j+k})) \\
        &= \H^{N-k}(\H^n(X_{j+k}))\\
        &= \H^{N-k-1}(\H^{n}(\H(X_{j+k})))\\
        &= \H^{N-k-1}(\H^{n}(Y_{j+k+1})).
    \end{align*}
    Since $\H^N(\H^n(X_j)) = \H^{N-1}(\H^n(\H(X_j))) = \H^{N-1}(\H^n(Y_{j+1}))$ shows the base case $k=1$ holds,
    the induction is complete.
    Therefore, one finds
    $\H^{N}(\H^n(X_j)) = \H^{0}(\H^n(Y_{j+N})) = \H^n(X_j)$,
    contradicting that $C$ is acyclic.
    Thus $i \mapsto L_i$ is injective.
    If there is a smallest foil $L_{i_0}$ then $i\mapsto L_{i_0+i}$ is an order
    isomorphism with $\N$,
    otherwise $i \mapsto L_i$ is an order isomorphism with $\Z$.
\end{proof}

\begin{lemma}\label{primevalrestriction}
    $\H$ restricts to a bijective point-shift $\H|_\n$ on the flow-adapted
    sub-process $\n:=\H^\infty(\m)$ of primeval elements.
\end{lemma}

\begin{proof}
    $\H$ naturally restricts to a
    point-shift $\H|_\n$ on $\n$ because if $X \in \H^\infty(\m)$ then
    $\H(X) \in \H^\infty(\m)$.
    By definition, primeval elements are in the image $\H(\m)$, but moreover
    they are in the image $\H(\n)$.
    Indeed, by \Cref{finitelymanychildren}, points in $\m$ have only finitely
    many children.
    If $X \in \n$ were such that none of its children were primeval, then there
    would be $n\in \N$ large enough that none of $X$'s children are in the
    image $\H^n(\m)$.
    But then $X$ would not be in $\H^{n+1}(\m)$, contradicting that $X \in
    \H^\infty(\m)$.
    Thus the restricted point-shift $\H|_\n$ is surjective.
    If $\n$ is not the empty process $\P$-a.s.\ then it has nonzero and finite intensity
    and $\E^\n[\Delta(\h^{-1})] =1$ by unimodularity so
    that surjectivity and injectivity are equivalent by
    \Cref{structuralpropertiesofpointshifts} (j), so $\H|_\n$ is bijective.
\end{proof}

The main result of this section follows.

\begin{theorem}[Cardinality Classification of a Component]\label{cardinalityclassification}
    $\P$-a.s.\ each connected component $C$ of $G^\H$ is in one of the three
    following classes:
    \begin{enumerate}
        \item \textbf{Class $\F/\F$:}
            $C$ is finite, and hence so is each of its $\H$-foils.
            In this case, when denoting by $1 \leq n = n(C) < \infty$
            the number of its foils:
            \begin{itemize}
                \item $C$ has a unique cycle of length $n$;
                \item $\H^\infty(\m)\cap C$ is the set of vertices of this cycle.
            \end{itemize}

        \item \textbf{Class $\I/\F$:}
            $C$ is infinite and each of its $\H$-foils is finite.
            In this case:
            \begin{itemize}
                \item $C$ is acyclic;
                \item Each foil has a junior foil;
                \item $\H^\infty(\m)\cap C$ is a unique \textdefn{bi-infinite} path,
                    i.e.\ a sequence $\{X_n\}_{n \in \Z}$ of points of $\m$ 
                    such that $\H(X_n) = X_{n+1}$ for all $n$.
            \end{itemize}

        \item \textbf{Class $\I/\I$:}
            $C$ is infinite and all its $\H$-foils are infinite.
            In this case:
            \begin{itemize}
                \item $C$ is acyclic;
                \item $\H^\infty(\m) \cap C = \emptyset$.
            \end{itemize}
    \end{enumerate}
\end{theorem}

\begin{proof}
    The properties of finite components $C$ are immediate, so only infinite components are considered.
    Recall that by \Cref{infinitecomponentsacyclic} $\P$-a.s.\ all infinite components are acyclic.
    Consider the collection $\mathfrak{N}$ of all infinite components that
    have both finite and infinite foils.
    Suppose $C \in \mathfrak{N}$.
    According to \Cref{finitelymanychildren}, all $X \in \m$ have only finitely
    many children, so that if $L$ is an infinite foil, then $L_+$
    is also infinite.
    It follows that there is a maximum finite foil $L$ with respect to the foil
    order in $C$.
    Let $\n\subseteq \bigcup \mathfrak{N}$ be the union of these maximum 
    finite foils of each $C \in \mathfrak{N}$.
    By construction, $\n \cap C$ is finite for each $C \in \mathfrak{N}$,
    so \Cref{noflowadaptedselection} implies $\n=\emptyset$ and hence
    $\mathfrak{N}=\emptyset$, $\P$-a.s.
    Thus $\P$-a.s.\ each infinite component is either of class $\I/\F$ or $\I/\I$.

    Next, redefine $\mathfrak{N}$ to be the set of infinite foils $L$ of $\m$, and
    let $\n:=\H^\infty(\m)$.
    By construction $\n \cap L$ is finite for each $L \in \mathfrak{N}$ because a foil
    cannot have multiple primeval elements.
    If $X\neq Y \in L$ were both primeval, then with $n$ minimal such that $\H^n(X) =
    \H^n(Y)$ one finds the primeval element $\H^n(X)$ is the image of two
    distinct primeval elements $\H^{n-1}(X),\H^{n-1} (Y)$, contradicting
    injectivity of $\H|_\n$ guaranteed by \Cref{primevalrestriction}.
    Thus \Cref{noflowadaptedselection} implies $\P$-a.s.\ $\n \cap L = \emptyset$ for all
    infinite foils $L$, and hence
    $\P$-a.s.\ $\H^\infty(\m) \cap C \neq \emptyset$ implies
    $C$ is of class $\I/\F$.

    Conversely, it will be shown that if $C$ is class $\I/\F$, then $\H^\infty(\m) \cap
    C\neq \emptyset$.
    Indeed, redefine $\mathfrak{N}$ to be the collection of components $C$ of class $\I/\F$
    that have a minimum foil in the foil order.
    Letting $\n \subseteq \bigcup \mathfrak{N}$ be the union of minimum foils in $C$,
    it holds that $\n \cap C$ is the (finite) minimum foil in $C$ for each $C \in
    \mathfrak{N}$.
    Thus \Cref{noflowadaptedselection} implies $\n=\emptyset$ and hence
    $\mathfrak{N}=\emptyset$, $\P$-a.s.
    Now consider a $C$ of class $\I/\F$
    and an arbitrary foil $L$ of $C$.
    Since $L$ is finite there is a minimum $n$ such that $\H^n(L)$ is a single
    point.
    Let $C_0$ denote the subgraph of $G^\H$ of $L$ together with all descendants of elements of
    $L$ and all forefathers of elements of $L$ up to $\H^n(L)$.
    Then $C_0$ is an infinite connected graph with vertices of finite degree,
    and hence it contains an infinite simple path $\{X_i\}_{i \leq 0}$ with
    $\H(X_i) = X_{i+1}$ for each $i<0$ by K\"onig's
    infinity lemma (c.f.\ Theorem 6 in \cite{konig1990theory}).
    For $i>0$, define $X_i := \H^i(X_0)$.
    Then $\{X_i\}_{i \in \Z}$ is a bi-infinite path in $C$ satisfying $\H(X_i)
    = X_{i+1}$ for all $i \in \Z$, and thus
    $\{X_i\}_{i \in \Z} \subseteq \H^\infty(\m) \cap C$, in particular showing
    $\H^\infty(\m) \cap C \neq \emptyset$.
    It also holds that $\H^\infty(\m) \cap C \subseteq \{X_i\}_{i \in \Z}$ since
    for any $X \in \H^\infty(\m) \cap C$ it is possible to choose $n,m$ such that $\H^n(X) =
    \H^m(X_0) = X_m$.
    Uniqueness of primeval children then implies $X = X_{m-n}$.
    It follows that $\H^\infty(\m) \cap C = \{X_i\}_{i \in \Z}$.

    Thus it is shown that $\P$-a.s.\ infinite components $C$ are class $\I/\F$
    if and only if $\H^\infty(\m) \cap C \neq \emptyset$ and in this case
    $\H^\infty(\m) \cap C$ is a unique bi-infinite sequence $\{X_i\}_{i \in
        \Z}$ satisfying $\H(X_i) = X_{i+1}$.
    Since $\I/\F$ and $\I/\I$ are the only possible choices,
    by process of elimination it follows that $\P$-a.s.\ infinite components $C$ are of class
    $\I/\I$ if and only if $\H^\infty(\m) \cap C = \emptyset$.
\end{proof}

\subsection{A Counterexample on a Non-unimodular Group}\label{sec:classification-counterexample}

This example serves to show that the cardinality classification (\Cref{cardinalityclassification}) does not
hold for non-unimodular spaces.
It is an open question whether a more general classification for such
spaces exists.
Recall the standard first example of a non-unimodular group: the $ax+b$
group.
In this section,
\[
    \X = \left\{\pmat{a & b\\0 & 1}: a>0, b \in \R\right\}
\]
with matrix
multiplication and the
topology inherited from $\R^4$.
$\X$ is identified with the right half-plane in $\R^2$ by identifying
$(a,b)$ with $\pmat{a&b\\0&1}$.
In this notation 
\[
    (a,b)(c,d) = (ac, ad+b), \qquad{} (a,b)^{-1} =
    (\frac{1}{a},-\frac{b}{a}).
\]
Then, cf.\ \cite{hewitt2012abstract} Example 15.17 (g), $\X$ has a left-invariant Haar measure
\[
    \Haar(B) = \iint_B \frac{1}{a^2}\,da\,db 
\]
and modular function
\[
    \Delta(a,b) = \frac{1}{a}.
\]
Let $\m$ be a homogeneous Poisson point process on $\X$ with intensity
$\intensity \in (0,\infty)$.
Necessarily $\m$ is $\X$-stationary and simple. 
For all $(a,b) \in \X$ define the strip 
\[
    S(a,b) := [a,\infty) \times [b-\delta a,b+\delta a]
\]
for some fixed $\delta>0$.
Note that the definition is chosen so $(a,b) S(1,0) = S(a,b)$, where
here $(1,0) = e \in \X$.
Moreover, for any $(a,b) \in \X$,
\[
    \Haar(S(a,b)) 
    = \int_{b-\delta a}^{b+\delta a} \int_{a}^\infty \frac{1}{x^2}\,dx\,dy
    = \frac{1}{a} \cdot ((b+\delta a)-(b-\delta a))
    = 2 \delta
\]
so in particular $\m(S(a,b)) < \infty$ a.s.
By the Slivnyak-Mecke theorem (see \Cref{slivnyakmecke} in the appendix), $\m^! := \m-\delta_e$ is Poisson under $\P^\m$ with
$\E^\m[\m^!(B)] = \intensity \Haar(B)$.
Hence $\E^\m[\m^!(S(1,0))] = 2\delta\intensity$ and therefore $\m(S(1,0)) < \infty$,
$\P^\m$-a.s.
Equivalently, $\P$-a.s.\ $\m(S(X)) < \infty$ for all $X \in \m$ by
\Cref{palmpalmostsureequivalence}.
This leads to the \textdefn{strip point-shift} $\H$ where $\H(X)$ is defined to be
the right-most point of $\m$ in $S(X)$ for each $X \in \m$.
\begin{figure}[t]
    \centering
    \includegraphics[width=\textwidth]{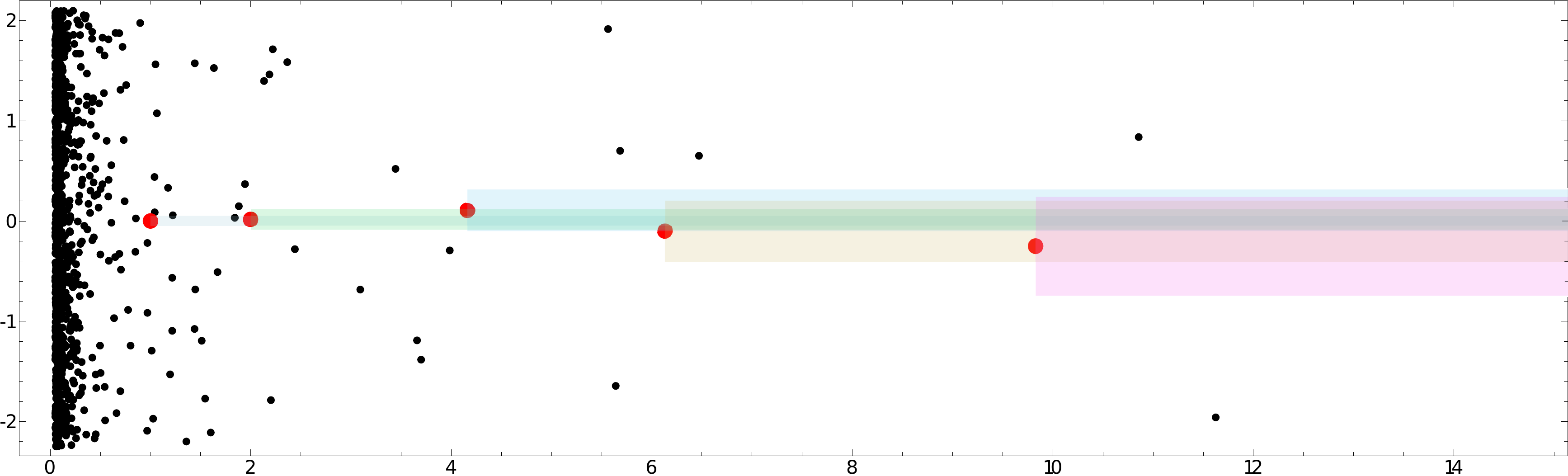}
    \caption{Iterates of $e$ under
        the strip point-shift on the $ax+b$ group.}
\end{figure}
One may theoretically resolve ties for right-most point using the lexicographic order on $\R^2$, but the interested
reader may note that
results in \Cref{sec:separating-points} will show that there is no need because almost surely each point $X\in\m$ has a unique first coordinate.

Now suppose that $2\delta\intensity < 1$.
It will be shown that $\P$-a.s.\ $\H^n(X)$ eventually becomes constant as $n \to
\infty$ for all $X \in \m$.
It suffices to show that under $\P^\m$ it holds that $\H^n(e)$ eventually becomes
constant.
Recall that the \textdefn{$n$-th factorial moment measure} of a counting
measure $\mu$ with representation $\mu = \sum_{i} \delta_{x_i}$ is
defined as $\mu^{(n)} := \sum_{i_1\neq \cdots \neq
    i_n} \delta_{(x_{i_1},\ldots,x_{i_n})}$, where the notation $i_1 \neq \cdots
\neq i_n$ means that $i_1,\ldots,i_n$ are all distinct.
Then
\begin{align*}
    &\E^\m\sum_{k=0}^\infty \m(S(\H^k(e))\setminus
    \{\H^{k}(e)\})\\
    & =\sum_{k=0}^\infty \E^\m \m(S(\H^k(e))\setminus
    \{\H^{k}(e)\})\\
    &\leq \sum_{k=0}^\infty \E^\m \int_{\X^{k+1}} 1_{x_1 \in S(e)} \cdots
        1_{x_{k+1}\in S(x_k)}\,(\m^!)^{(k+1)}(dx_1 \times \cdots \times dx_{k+1})\\
    &=\sum_{k=0}^\infty \E\int_{\X^{k+1}} 1_{x_1 \in S(e)} \cdots
    1_{x_{k+1}\in S(x_k)}\,\m^{(k+1)}(dx_1 \times \cdots \times dx_{k+1})\\
    &=\sum_{k=0}^\infty \int_{\X^{k+1}} 1_{x_1 \in S(e)} \cdots
    1_{x_{k+1}\in S(x_k)} \intensity^{k+1}\,\Haar(dx_{k+1})\cdots \Haar(dx_1)\\
    &=\sum_{k=0}^\infty (2\delta)^{k+1} \intensity^{k+1}\\
    &< \infty,
\end{align*}
where here the Slivnyak-Mecke theorem is used again, along with the fact that the factorial
moment measures of a Poisson point
process are just powers of the intensity measure, cf.\ Example 9.5 (d) in~\cite{daley2007introduction}.
Thus it must be that $\m(S(\H^k(e))\setminus\{\H^k(e)\}) = 0$ for all $k$
large, $\P^\m$-a.s.
That is, there are no points of $\m$ in $S(\H^k(e))$ besides $\H^k(e)$ itself.
Consequently, $\H^k(e)$ is a fixed point of $\H$ for large $k$ and $\H^k(e)$ is thus
eventually constant in $k$.
Equivalently, $\P$-a.s.\ for every $X \in \m$ it holds that $\H^k(X)$ is eventually
constant in $k$.

Next it will be shown that every fixed point of $\H$ is the image of infinitely
many $X \in \m$.
Again it is enough to show under $\P^\m$ that if $\H(e) = e$ then $e$ is the
image of infinitely many $X \in \m$.
This is accomplished by finding a region of points $(x,y)\in \X$ such that 
\begin{enumerate}[label=(\roman*)]
    \item  $(1,0) \in S(x,y)$, and
    \item $S(x,y) \cap ([1,\infty)\times \R) \subseteq S(1,0)$,
\end{enumerate}
which implies $\H$ would map a point of $\m$ at $(x,y)$ to $(1,0)$.
The condition (i) says $1 \geq x$ and $y-\delta x \leq 0
\leq y+\delta x$,
i.e.\ $-\delta x \leq y \leq \delta x$.
Condition (ii) is guaranteed if $[y-\delta x,y+\delta x] \subseteq
[-\delta,\delta]$, i.e.\
if $y \geq \delta(x-1)$ and $y \leq \delta(1-x)$.
The constraints
\begin{equation*}
    0 < x \leq 1,\quad -\delta x \leq y \leq \delta x,\quad y \leq
    \delta(1-x),\quad y \geq \delta(x-1),
\end{equation*}
bound a parallelogram $D$ with corners
\[
    (0,0),\quad (1/2,\delta/2),\quad (1,0),\quad (1/2,-\delta/2). 
\]
Then
\[
    \E^\m[\m^!(D)] = \intensity \Haar(D) \geq \intensity\int_0^{1/2}\int_{-\delta
        x}^{\delta x} \frac{1}{x^2}\,dy\,dx
    = \intensity\int_0^{1/2}\frac{2\delta}{x}\,dx = \infty
\]
so that the region $D$ contains infinitely many points of $\m$, $\P^\m$-a.s.
By construction, if $\H(e)=e$ then every $X \in \m \cap D$ has $\H(X) = e$, 
proving the claim.

Putting previous claims together, it holds that the foils and connected
components are identical because every component contains a fixed point,
and the foils and components are in bijection with the fixed points of
$\H$.
The connected component of a fixed point $Y$ of $\H$ is all $X \in \m$ that
are eventually sent to $Y$.
Thus all components and foils are infinite (class $\I/\I$).
However, the components are not
acyclic and $\H^\infty(\m) = \{X \in \m : \H(X) = X\} \neq \emptyset$,
contrary to what the
classification theorem would suggest for unimodular $\X$.
It follows that the properties of the cardinality classification cannot be
extended beyond the case of unimodular $\X$.

\section{Properties of Point-shifts}

\subsection{Mecke's Invariance Theorem}\label{sec:mecke-invariance-thm}

In the case of $\X=\R^d$, Mecke's invariance theorem shows that Palm probabilities 
are preserved under bijective point-shifts.
Even stronger, a point-shift is bijective if and only if it preserves Palm probabilities.
It will be shown in \Cref{meckethm} that if $\X$ is unimodular then this still
holds.
However, for non-unimodular $\X$ this is not so.
Precisely, the notion of isomodularity defined in the introduction will be elaborated upon, and
it will be shown that, amongst
bijective point-shifts, isomodular ones are exactly those that preserve Palm
probabilities (\Cref{isomodularinvariance}).

For the rest of the section, fix a flow-adapted simple point process 
$\m$ of intensity $\intensity \in (0,\infty)$, and a point-map $\h$
with associated point-shift $\H$.
The notation for the corresponding functions $\tau^\H$,
$h^+$, $h^-$, $H^+$, $H^-$, $\h^-$, $\H^-$ mentioned in the preliminaries is retained.

The simple case of Mecke's invariance theorem when $\X$ is unimodular follows.

\begin{corollary}[Mecke's Invariance Theorem]\label{meckethm} 
    Suppose that $\X$ is unimodular.
    Then $\H$ preserves $\P^\m$ if and only if $\H$ is bijective.
    That is, $\P^\m(\theta_\h^{-1} \in A) = \P^\m(A)$ for all $A \in \A$ if and only if
    $\H$ is bijective.
\end{corollary}

\begin{proof}
    Apply \Cref{structuralpropertiesofpointshifts} (f), the test for bijectivity, and use the fact that $\Delta(x) = 1$ for
    all $x \in \X$.
\end{proof}

With Mecke's invariance theorem for unimodular $\X$ in place, one may ask about
non-unimodular $\X$. 
For these $\X$, which bijective point-shifts preserve Palm probabilities?
\Cref{bijectivitytest} shows that the obstruction is the factor
$\Delta(\h^{-1})$.
This motivates the definition of isomodularity, which says that a point-shift 
preserves the value of $\Delta(X)$ for each $X \in \m$.
Isomodularity, defined already in the introduction, is a special case of invariance of a subgroup under $\H$, which
is defined presently.

\begin{definition}
    A measurable subgroup $G \in \B(\X)$ of $\X$ is called
    \textdefn{$\H$-invariant} if $\P$-a.s.\ $\H(X)$ is in the same coset as $X$ for
    all $X \in \m$.
\end{definition}

Isomodularity of $\H$ is the same as the assumption that the subgroup $\{\Delta=1\}$ is $\H$-invariant.
Also note that if $\X$ is unimodular, then $\H$ is automatically isomodular.

A brief detour is taken to go through the equivalent descriptions of $\H$-invariance under $\P$
and $\P^\m$.

\begin{proposition}\label{preservesubgroupequivalences}
    Let $G\in\B(\X)$ a measurable subgroup of $\X$,
    and for each $x\in \X$ let $[x] := xG$ denote the coset of $x$.
    Then the following are equivalent
    \begin{enumerate}[label=(\alph*)]
        \item $G$ is $\H$-invariant, i.e.\ $\P$-a.s.\ $[\H(X)] = [X]$
            for all $X \in \m$,
        \item $\P^\m$-a.s.\ $[\h] = [e]$,
    \end{enumerate}
    and if $\H$ is bijective, the previous statements are also equivalent to
    \begin{enumerate}[label=(\alph*)]\setcounter{enumi}{2}
        \item $\P$-a.s.\ $[\H^-(X)] = [X]$ for all $X \in \m$,
        \item $\P^\m$-a.s.\ $[\h^-] = [e]$.
    \end{enumerate}
\end{proposition}

\begin{proof}
    \begin{enumerate}
        \item[]
        \item[] \emph{(a)}$\iff$\emph{(b)}:
            The equivalence follows from \Cref{palmpalmostsureequivalence},
            so that $\P^\m$-a.s.\ $[\h] = [e]$ is equivalent to 
            $\P$-a.s.\ $[\h(\theta_X^{-1})] = [e]$ for all $X \in \m$,
            which is the same as $[\H(X)] = [X]$ after multiplying
            by $X$.
        \item[] \emph{(a)}$\iff$\emph{(c)}:
            Using that $\H$ and $\H^-$ are inverses,
            replace $X$ with $\H^-(X)$ in \emph{(b)} to get \emph{(c)}
            or replace $X$ with $\H(X)$ in \emph{(c)} to get \emph{(b)}.
        \item[] \emph{(c)}$\iff$\emph{(d)}:
            The proof is the same as \emph{(a)}$\iff$\emph{(b)}. 
    \end{enumerate}
\end{proof}

Since isomodularity plays an important role in what follows, the previous
result is restated for $G:=\{\Delta=1\}$ in the bijective case.

\begin{corollary}\label{isomodularequivalences}
    Let $\H$ be bijective, then the following are equivalent
    \begin{enumerate}[label=(\alph*)]
        \item $\H$ is isomodular, i.e. $\P$-a.s.\ $\Delta(\H(X)) = \Delta(X)$
            for all $X \in \m$,
        \item $\P^\m$-a.s.\ $\Delta(\h) = 1$,
        \item $\P$-a.s.\ $\Delta(\H^-(X)) = \Delta(X)$ for all $X \in \m$,
        \item $\P^\m$-a.s.\ $\Delta(\h^-) = 1$.
    \end{enumerate} 
\end{corollary}

Now the question of which bijective point-shifts preserve
Palm probabilities is answerable.

\begin{theorem}
    \label{isomodularinvariance}
    Suppose $\H$ is bijective.
    Then $\H$ preserves $\P^\m$ if and only if $\H$ is isomodular.
    That is, $\P^\m(\theta_\h^{-1} \in A) = \P^\m(A)$ for all $A \in \A$ 
    if and only if $\H$ is isomodular.
\end{theorem}

\begin{proof}
    Suppose $\H$ is isomodular.
    Then $\Delta(\h^-) = 1$, $\P^\m$-a.s.\ by \Cref{isomodularequivalences}.
    Hence \eqref{changeofpalm} immediately implies
    $\H$ preserves $\P^\m$.
    If $\H$ is not isomodular,
    at least one of $\P^\m(\Delta(\h^-) > 1)$ and $\P^\m(\Delta(\h^-) < 1)$ is strictly positive.
    The cases are nearly identical, so assume $\P^\m(\Delta(\h^-)>1)>0$
    and take $A:= \{\Delta(\h^-) > 1\}$.
    Then take $f := 1_{A}$ in \eqref{changeofpalm}
    to find 
    \[
        \P^\m(\theta_\h^{-1} \in A) 
        = \E^\m\left[\frac{1_{\Delta(\h^-) > 1}}{\Delta(\h^-)}\right]
        < \E^\m\left[1_{\Delta(\h^-) > 1}\right]
        = \P^\m(A),
    \]
    showing that $\P^\m$ is not preserved. 
\end{proof}

\subsection{Reciprocal and Reverse of a Point-map}\label{sec:reciprocal-and-reverse}
A curious interplay between the reverse $\h^-$ and the
reciprocal $\h^{-1}$ of a point-map is investigated, and a characterization of when the two
have the same law under $\P^\m$ is given.

The notation of the previous section is retained.
That is, $\m$ is a flow-adapted simple point process of intensity $\intensity \in (0,\infty)$,
and $\h$ is a point-map
with associated point-shift $\H$.
The notation for the corresponding $\tau^\H$,
$h^+$, $h^-$, $H^+$, $H^-$, $\h^-$, $\H^-$ defined in the preliminaries is also retained.
Next follows another result along the lines of \Cref{structuralpropertiesofpointshifts} (f) and (g) 
which sparks interest in the distributional relationship between $\h^{-1}$ and $\h^-$.

\begin{corollary}\label{changeofhthm}
    Suppose $\H$ is bijective.
    For all $f:\X\to \R_+$ measurable it holds that
    \begin{equation}\label{changeofh0}
        \E^\m\left[f(\h^{-1})\Delta(\h^{-1})\right] =
        \E^\m\left[f(\h^-)\right],
    \end{equation}
    \begin{equation}\label{changeofh}
        \E^\m\left[f(\h^{-1})\right] = \E^\m\left[\frac{f(\h^-)}{\Delta(\h^-)}\right].
    \end{equation}
\end{corollary}

\begin{proof}
    Use the fact that $\P^\m$-a.s.\
    \[
        \h^-(\theta_\h^{-1}) = \h^{-1}(\h \h^-(\theta_\h^{-1})) = \h^{-1} \H^-(\h) =\h^{-1} \H^-(\H(e))= \h^{-1}\]
    and replace $f$ by $f(\h^-)$ in each of 
    \eqref{changeofpalm0} and \eqref{changeofpalm}.
\end{proof}

One sees in \eqref{changeofh} that non-unimodularity of $\X$ is, as usual, an
obstruction. Two more results relating the distributions of
$\Delta(\h^-)$ and $\Delta(\h^{-1})$ are given.
Then it is shown in \Cref{hinversehminussamelaw} that amongst 
bijective point-shifts, the isomodular ones are
precisely those for which $\h^{-1}$ and $\h^-$ have the same distribution under
$\P^\m$.
Recall that this is also the class of point-shifts that preserve Palm
probabilities by \Cref{isomodularinvariance}.

\begin{corollary}
    Let $\H$ be bijective, then for all $r>0$ it holds that
    \[
        r\P^\m(\Delta(\h^{-1})=r) = \P^\m(\Delta(\h^-)=r),
    \]
    and if this number is strictly positive then for all $A \in \A$
    \begin{equation*}
        \P^\m(\theta_{\h}^{-1} \in A\mid \Delta(\h^{-1})=r) = \P^\m(A \mid \Delta(\h^-)=r).
    \end{equation*}
\end{corollary}

\begin{proof}
    Fix $r>0$ and take $f(x):= 1_{\Delta(x) = r}$ in \eqref{changeofh}.
    One finds 
    \[
        \P^\m(\Delta(\h^{-1})=r) = \frac{1}{r}\P^\m(\Delta(\h^-)=r) =: p
    \]
    showing the first claim.
    Supposing that $p>0$, take
    $f:=1_{A} 1_{\Delta(\h^-)=r}$ in \eqref{changeofpalm}
    and use that $\P^\m$-a.s.\ $\h^-(\theta_\h^{-1}) = \h^{-1}$ to find
    \[
        \P^\m(\theta_\h^{-1} \in A, \Delta(\h^{-1})=r) 
        = \frac{1}{r}\P^\m(A, \Delta(\h^-)=r).
    \]
    Division by $p$ finishes the proof. 
\end{proof}

\begin{lemma}
    Let $\H$ be bijective, then for all $\alpha \in \R$ and $0\leq r \leq s \leq \infty$
    it holds that
    \begin{equation}\label{deltahpowers}
        \E^\m\left[ \Delta(\h^{-1})^{\alpha}1_{r \leq \Delta(\h^{-1}) \leq s} \right] 
        = \E^\m\left[ \Delta(\h^-)^{\alpha-1}1_{r \leq \Delta(\h^-) \leq s}\right].
    \end{equation}
\end{lemma}

\begin{proof}
    Take $f(x) := \Delta(x)^\alpha 1_{r \leq \Delta(x)\leq s}$ in \eqref{changeofh}.
\end{proof}

\begin{theorem}\label{hinversehminussamelaw}
    Let $\H$ be bijective,
    then $\h^{-1}$ and $\h^-$ have the same law under $\P^\m$
    if and only if $\H$ is isomodular.
\end{theorem}

\begin{proof}
    Suppose $\H$ is isomodular.
    Then by \Cref{isomodularequivalences}, 
    $\P^\m$-a.s.\ $\Delta(\h)=\Delta(\h^-) =1$ and thus
    \eqref{changeofh} shows that $\h^{-1}$ and $\h^-$ have the same
    law under $\P^\m$.

    Next suppose that $\h^{-1}$ and $\h^-$ have the same law under $\P^\m$.
    Then 
    \begin{equation}\label{hdeltaind}
        \E^\m[\Delta(\h^{-1})^\alpha1_{r \leq \Delta(\h^{-1}) \leq s}] 
        = \E^\m[\Delta(\h^-)^\alpha1_{r \leq \Delta(\h^-) \leq s}]
    \end{equation}
    for
    all $\alpha \in \R$ and all $0 \leq r \leq s \leq \infty$.
    But then for all $\alpha \in \R$ and all $0 \leq r \leq s \leq \infty$
    \begin{align*}
        \E^\m[\Delta(\h^{-1})^{\alpha+1}1_{r \leq \Delta(\h^{-1}) \leq s}] 
        &= \E^\m[\Delta(\h^-)^\alpha1_{r \leq \Delta(\h^-) \leq s}] 
        & \text{(by \eqref{deltahpowers})} \\
        &= \E^\m[\Delta(\h^{-1})^{\alpha}1_{r \leq \Delta(\h^{-1}) \leq s}]
        & \text{(by \eqref{hdeltaind})}\\
        &= \E^\m[\Delta(\h^-)^{\alpha-1}1_{r \leq \Delta(\h^{-1}) \leq s}].
        & \text{(by \eqref{deltahpowers})}
    \end{align*}
    Taking $\alpha:=1, r:=1, s:= \infty$
    \[
        \E^\m[\Delta(\h^{-1})^{2}1_{1 \leq \Delta(\h^{-1})}] 
        = \E^\m[\Delta(\h^{-1})1_{1 \leq \Delta(\h^{-1})}]
    \]
    which is absurd unless $\Delta(\h^{-1}) \leq 1$, $\P^\m$-a.s.
    It also holds that with  $\alpha:=1, r:= 0, s:= 1$,
    \[
        \E^\m[\Delta(\h^{-1})^{2}1_{\Delta(\h^{-1}) \leq 1}] 
        = \E^\m[\Delta(\h^{-1})1_{\Delta(\h^{-1}) \leq 1}],
    \]
    which is absurd unless $\Delta(\h^{-1}) \geq 1$, $\P^\m$-a.s.
    It follows that $\Delta(\h^{-1}) =1$, $\P^\m$-a.s.
    By \Cref{isomodularequivalences} the result follows. 
\end{proof}

\subsection{Separating Points of a Point Process}\label{sec:separating-points}

In this section a notion of a
function separating points of a point process is introduced.
For the remainder of the section, $\m$ is a simple and flow-adapted point process of intensity $\intensity \in (0,\infty)$.

\begin{definition}
    Let $S$ be a set, $f:\X \to S$, and suppose
    that $\P$-a.s.\ no distinct $X,Y \in \m$ have $f(X) = f(Y)$.
    Then say that $f$ \textdefn{separates points} of $\m$.
    Similarly, say that a fixed partition $\{B_i\}_{i \in J}$ of $\X$
    \textdefn{separates points} of $\m$ if $\P$-a.s.\ no
    $B_i$ contains more than $1$ point of $\m$.
\end{definition}

When separation of points occurs is studied by proving a general result
concerning 
when there cannot be an $n$-tuple of distinct points of $\m$ satisfying a
given constraint.
Recall again that $\mu^{(n)} = \sum_{i_1 \neq \cdots \neq i_n} \delta_{(x_{i_1},\ldots,x_{i_n})}$ denotes the $n$-th factorial moment measure of a measure $\mu = \sum_i \delta_{x_i}$,
and $\mu^! = \mu - \delta_e$.

\begin{theorem}\label{generalconditionforseparatingpoints}
    Let $(S,\Sigma)$ be a measurable space and fix $M \in \Sigma$.
    Let $F:\X\times \X^n \to S$ be measurable, and
    suppose that for all $y=(y_1,\ldots,y_n) \in (\X\setminus \{e\})^n$,
    or more generally that for $\E^\m[(\m^!)^{(n)}]$-a.e.\ $y \in \X^n$, 
    \[
        \Haar(x \in \X : F(x,xy) \in M) = 0.
    \]
    Then $\P$-a.s.\ no $n+1$ distinct $X,Y_1,\ldots,Y_n\in \m$ have
    $F(X,Y_1,\ldots,Y_n) \in M$.
\end{theorem}

\begin{proof}
    By straight calculations,
    \begin{align*}
        &\P(\exists X \in \m, Y \in \m^{(n)}: (X,Y) \in \m^{(n+1)}, F(X,Y) \in M)\\
        &\leq \E \int_\X 1_{\exists Y \in \m^{(n)}: \forall i,
            Y_i \neq x, F(x,Y) \in M}\,\m(dx)\\
        &\leq \E \int_\X \m^{(n)}(\theta_e,\{y \in \X^n: \forall
        i, y_i  \neq x, F(x,y) \in M\})\, \m(dx)\\
        &=\intensity\E^\m\int_\X \m^{(n)}(\theta_x,\{y \in
        \X^n,\forall i, y_i \neq x, F(x,y) \in M\})\,\Haar(dx)\\
        &=\intensity\E^\m\int_\X \m^{(n)}(\theta_e,\{x^{-1}y: y \in \X^n,
        \forall i, y_i \neq x, F(x,y) \in M\})\,\Haar(dx)\\
        &=\intensity\E^\m\int_\X \m^{(n)}(\theta_e,\{y\in \X^n, \forall i, xy_i \neq
        x, F(x,xy)\in M\})\,\Haar(dx)\\
        &=\intensity\E^\m\int_\X\int_\X 1_{F(x,xy) \in
            M}\,(\m^!)^{(n)}(dy)\,\Haar(dx)\\
        &=\intensity\E^\m\int_\X \Haar(x \in \X : F(x,xy)\in
        M)\,(\m^!)^{(n)}(dy)\\
        &= 0,
    \end{align*}
    where in the third equality the refined Campbell theorem, stated in the appendix as~\Cref{CLMM}, is used.
    This proves the claim.
\end{proof}

\Cref{generalconditionforseparatingpoints} immediately gives a condition for
separating points of $\m$.

\begin{corollary}[Condition for Separating Points]\label{conditionforseparatingpoints}
    Let $(S,\Sigma)$ be a measurable space, $f:\X \to S$ measurable, and suppose
    for all $y \neq e$, or more generally for $\E^\m[\m^!]$-a.e.\ $y\in \X$,
    \[
        \Haar(x \in \X : f(x) = f(xy)) = 0.
    \]
    Then $f$ separates points of $\m$. Implicit in the previous line is the assumption 
    that the sets $\{x \in \X: f(x) = f(xy)\}$ are measurable for all $y \in \X$.
    This is automatic if $(S,\Sigma)$ is a standard measurable space, or more
    generally if $S\times S$ has measurable diagonal.
\end{corollary}

\begin{proof}
    Take $n:=1$, $F(x,y) := (f(x),f(y))$ for all $x,y \in \X$, and
    take $M$ to be the diagonal of $S \times S$, then apply
    \Cref{generalconditionforseparatingpoints}.
\end{proof}

\Cref{conditionforseparatingpoints} generalizes the well-known theorem in 
$\X = \R^d$
that a stationary point process has not two points equidistant from $0$.
That would be the case of $f(x) := |x|$.
Not all $\X$ have this property though.
Indeed, if $\X$ is a countable group with the discrete distance $d(x,y) :=
1_{x\neq y}$, then
$\Haar(x \in \X: d(x,e) = d(xy,e))>0$ for all $y \neq e$
so the result does not apply if $\X$ has more than one element.

The next results can be used to show that there is no need to resolve ties when defining a point-shift
in some situations. Intuitively, if a set $B$ is small from the typical point's perspective,
then no shift of $B$ will contain more than one point of $\m$.

\begin{proposition}\label{separatepointsbyshiftedsets}
    Let $B \in \B(\X)$ with $e \in B$.
    If $\E^\m[\m^!(B)] = 0$, then $\P$-a.s.\ for all $X \in \m$ it holds that $\m(Xb: b \in B)
    = 1$, i.e.\ $X$ is the unique point of $\m$ inside $\{Xb : b \in B\}$.
\end{proposition}

\begin{proof}
    The hypotheses imply $\P^\m$-a.s.\ $\m(B\setminus\{e\}) = 0$.
    By \Cref{palmpalmostsureequivalence}, $\P$-a.s.\ all $X \in \m$
    are such that $T_{X^{-1}}\m(B\setminus\{e\}) = 0$, i.e. $\m(\{Xb:b \in B\} \setminus
    \{X\}) = 0$, and hence $\m(Xb: b \in B) = 1$.
\end{proof}

For example, recall the strip point-shift on the $ax+b$ group of \Cref{sec:classification-counterexample}.
It was defined by sending a point $X$ to the right-most point in a certain strip in the plane.
In that case, take $B:= \{1\} \times \R$ in the previous result to find that the points of $\m$ have unique first coordinates.
Hence there are no ties for right-most point.

Finally, the previous result is restated in the case that $B$ is a subgroup and applied to see
that the only way for $\H$ to preserve a small subgroup from the typical point's perspective
is to act as the identity.

\begin{corollary}\label{distinctcosets}
    Let $G\in \B(\X)$ a subgroup of $\X$.
    If $\E^\m[\m^!(G)] = 0$, then the cosets of $G$ separate points of $\m$.
    \qed{}
\end{corollary}

\begin{corollary} \label{isoisidentity}
    Let $G \in \B(\X)$ a subgroup of $\X$.
    If $\E^\m[\m^!(G)] = 0$ but $G$ is $\H$-invariant for some point-shift $\H$,
    then $\H$ is the identity point-shift $\P$-a.s.
\end{corollary}

\begin{proof}
    $G$ being $\H$-invariant means 
    $\H(X)$ and $X$ are in the same coset for $X \in \m$, then
    by \Cref{distinctcosets} $\H$ is the identity point-shift. 
\end{proof}

\begin{corollary} \label{poissonisoisidentity}
    Let $G \in \B(\X)$ a subgroup of $\X$.
    If $\Haar(G) = 0$ and $\m$ is Poisson with intensity $\intensity \in (0,\infty)$,
    then the only $\H$ for which $G$ is $\H$-invariant is the identity.
\end{corollary}

\begin{proof}
    The Slivnyak-Mecke theorem, \Cref{slivnyakmecke} in the appendix,
    implies that $\E^\m[\m^!(G)] = \intensity \Haar(G) = 0$ and
    \Cref{isoisidentity} applies. 
\end{proof}

\section{Connections with Unimodular Networks}\label{sec:connections-unimodular-networks}

This section investigates the relationship between vertex-shifts on
unimodular random networks and point-shifts of $\X$-stationary point processes when $\X$ is unimodular.
See \cite{aldous2007processes} for a general reference on
unimodular networks, and \cite{baccelli2016networks} for a similar investigation
to this one for point processes on $\R^d$.

A graph $\G$ with vertex set $\V$ and undirected edge set $\EE$ is
written $\G = (\V,\EE)$.
Write $\V(\G)$ and $\EE(\G)$ for the vertices and edges of $\G$.
A \textdefn{network} is a graph $\G=(\V,\EE)$ together with a complete
separable metric space $\Xi$ called the \textdefn{mark space} and maps $\xi_\V:\V
\to \Xi$ and $\xi_\EE:\{(\v,\e) \in \V\times \EE:\v\sim \e\}\to \Xi$ 
corresponding to vertex and edge marks, i.e.\ extra information attached
to vertices and edges.
Throughout this document, all networks are assumed to be connected and locally finite,
i.e.\ $\G$ is connected and all vertices of $\G$ have finite degree.
Graphs are special cases of networks that have every mark equal to some
constant.
For $r \geq 0$, the ball (with respect to graph distance) of 
radius $\lceil r \rceil$ with center $\v \in
\V(\G)$ is denoted $N_r(\G,\v)$.

An \textdefn{isomorphism} of networks $\G_1=(\V_1,\EE_1)$ and $\G_2=(\V_2,\EE_2)$
with mark space $\Xi$ (one may always assume a common mark space $\N^\N$, or any
other fixed Polish space) 
is a pair of bijections
$\varphi_\V:\V_1\to\V_2$, and $\varphi_\EE:\EE_1\to\EE_2$ such that
\begin{enumerate}[label=(\roman*)]
    \item $\{\v, \v'\} \in \EE_1$ if and only if $\{\varphi_\V(\v),\varphi_\V(\v')\}
        \in \EE_2$,
    \item $\varphi_\EE(\{\v, \v'\}) =
        \{\varphi_\V(\v),\varphi_\V(\v')\}$ for all $\{\v,\v'\} \in \EE_1$,
    \item the vertex mark maps $\xi_{\V_1},\xi_{\V_2}$ satisfy
        $\xi_{\V_1} = \xi_{\V_2}\circ\varphi_\V$, and
    \item the edge mark maps $\xi_{\EE_1},\xi_{\EE_2}$ satisfy
        $\xi_{\EE_1}= \xi_{\V_2}\circ\varphi_\EE$.
\end{enumerate}

A \textdefn{rooted network} is a pair $(\G,\o)$ in which $\G$ is a network and $\o$
is a distinguished vertex of $\G$ called the \textdefn{root}.
An \textdefn{isomorphism} of rooted networks $(\G,\o)$ and $(\G',\o')$
is a network isomorphism such that $\varphi_\V(\o) = \o'$.
Let $\calG$ be the set of isomorphism classes of networks, 
and let $\calGs$ be the set of isomorphism classes of rooted networks.
Similarly define $\calGss$ for networks with a pair of distinguished vertices.
The isomorphism class of a network $\G$ (resp.\ $(\G,\o)$ or $(\G,\o,\v)$) is
denoted $[\G]$ (resp.\ $[\G,\o]$ or $[\G,\o,\v]$).

Equip $\calGs$ (and similarly for $\calGss$) with a metric and its Borel
$\sigma$-algebra.
The distance between $[\G,\o]$ and $[\G',\o']$ is $2^{-\alpha}$, where $\alpha$
is the supremum of those $r>0$ such that there is a rooted isomorphism between
$N_r(\G,\o)$ and $N_r(\G',\o')$ such that the distance of the marks of the
corresponding elements is at most $\frac{1}{r}$.
Then $\calGs$ is a complete separable metric space, and measurable functions on
$\calGs$ are those that can be identified by looking at finite neighborhoods of
the root.

A \textdefn{random network} is a random element in $\calGs$.
That is, it is a measurable map from $(\Omega,\A,\P)$ to
$\calGs$.
This map will be denoted $[\GG, \oo]$.
A random network $[\GG,\oo]$ is \textdefn{unimodular} if for all measurable
$g:\calGss \to \R_+$, the following mass transport principle is satisfied,
\[
    \E \sum_{\v \in \V(\GG)} g[\GG,\oo,\v] 
    = \E\sum_{\v \in \V(\GG)}g[\GG,\v,\oo].
\]

A \textdefn{vertex-shift}, which is the analog of a point-shift, is a map $f$ that associates to each network $\G$ a
function $f_\G:\V(\G) \to \V(\G)$ such that
\begin{enumerate}[label=(\roman*)]
    \item for all isomorphic networks $\G'$ with vertex isomorphism
        $\varphi_\V:\V(\G) \to \V(\G')$, one has 
        $f_{\G'} \circ \varphi_\V = \varphi_\V \circ f_\G$, and
    \item $[\G,\o,\v] \mapsto 1_{f_\G(\o)=\v}$ is measurable on
        $\calGss$.
\end{enumerate}

A general situation is given under which a point-process, seen under its Palm probability measure and rooted at the identity,
is a unimodular network.
First, the appropriate notion of flow-adaptedness for networks must be given,
then the result follows.

\begin{definition}
    Suppose $\beta$ is a map on $\Omega$ such that for all $\omega \in \Omega$, $\beta(\omega)$
    is a network whose vertex set $V(\beta) \subseteq \X$.
    Then $\beta$ is called \textdefn{flow-adapted} if
    for all $z \in \X$, $\beta(\theta_z\omega)$ is the \textdefn{shift} $T_{z}\beta(\omega)$ of $\beta(\omega)$ by $z$, i.e.\
    $\V(\beta(\theta_z\omega)) = \{zX: X \in \V(\beta(\omega))\}$, and
    $E(\beta(\theta_z\omega)) = \{ \{zX, zY\}: \{X,Y\} \in E(\beta(\omega))\}$,
    and all marks are preserved.
\end{definition}

\begin{theorem}\label{ppisunimodularnetwork}
    Suppose $\X$ is unimodular.
    Let $\m$ be a flow-adapted simple point process with intensity
    $\intensity\in(0,\infty)$.
    Let $\beta$ be a map on $\Omega$ such that for all $\omega\in \Omega$, $\beta(\omega)$ is a network,
    and such that $\beta$ satisfies
    \begin{enumerate}[label=(\roman*)]
        \item $\V(\beta) = \m$,
        \item $\beta$ is flow-adapted,
        \item $\omega \mapsto [\beta(\omega),e]$ is measurable on the
            event $\{e \in \V(\beta)\}$.
    \end{enumerate}
    Then $[\beta,e]$ is a unimodular network under $\P^\m$ on the event
    $\{ e \in \V(\beta)\} = \{e \in \m\}$.
\end{theorem}

\begin{proof}
    The assumption (iii) implies $[\beta,e]$ is a random network under $\P^\m$ on $\{e \in \V(\beta)\}$, so one only needs to check unimodularity.
    Let $g:\calGss\to \R_+$ be given.
    Then
    \begin{align*}
        \E^\m\sum_{\v \in \V(\beta)} g[\beta,e,\v]
        &=\E^\m \int_\X g[\beta,e, x]\,\m(dx)\\
        &=\E^\m \int_\X g[\beta(\theta_y^{-1}),e, y^{-1}]\,\m(dy) 
        & \text{(mass transport)}\\
        &=\E^\m \int_\X g[T_{y^{-1}}\beta,e, y^{-1}]\,\m(dy)\\
        &=\E^\m \int_\X g[\beta,y, e]\,\m(dy)\\
        &=\E^\m\sum_{\v \in \V(\beta)} g[\beta,\v,e],
    \end{align*}
    showing unimodularity.
\end{proof}

Recall that the symbol $T_z$ for $z \in \X$, used in the previous result as the shift operator on networks,
is also used as the shift operator on counting measures and point processes, which is how it is used in the following.

\begin{definition}\label{Xembeddingdefn}
    Let $[\GG,\oo]$ be a unimodular network.
    An \textdefn{$\X$-embedding} of $[\GG,\oo]$ with respect to a probability measure $\mathcal{P}$ on $\Omega$ 
    is a map
    $\embedding:\calGs\to\MM$ with the following properties:
    \begin{enumerate}[label=(\roman*)]
        \item $\embedding$ is measurable,
        \item $e \in \embedding[\G,\o]$ for all rooted networks $(\G,\o)$,
        \item there is a measurable function $\t:\calGss \to \X$ such
            that for every rooted network $(\G,\o)$ and every vertex $\v \in
            \V(\G)$,
            \begin{align*}
                \embedding[\G,\v] &=
                T_{\t[\G,\o,\v]^{-1}}\embedding[\G,\o],\\
                \t[\G,\v,\o] &= \t[\G,\o,\v]^{-1},
            \end{align*}
        \item $\mathcal{P}$-almost surely, the map defined on $\V(\GG)$ by $\v \mapsto
            \t[\GG,\oo,\v]$ is a bijection between $\V(\GG)$ and the support of 
            $\embedding[\GG,\oo]$.
    \end{enumerate}
    Say that the Palm version of a point process $\m$ is an \textdefn{$\X$-embedding} of
    $[\GG,\oo]$ if there is an $\X$-embedding $\embedding$ with respect to $\P^\m$ such that
    $\P^\m$-a.s.\ $\m=\embedding[\GG,\oo]$.
\end{definition}

Two of the motivating open questions of this research are:
\begin{enumerate}
    \item{} For a fixed $\m$, is the Palm version of $\m$ an $\X$-embedding of some $[\GG,\oo]$?
    \item{} For a fixed $[\GG,\oo]$, is there an $\m$ such that the Palm version of $\m$ is an $\X$-embedding of $[\GG,\oo]$?
\end{enumerate}

When $\X$ is unimodular, the answer to the first question is ``yes'' for $\m$ if it is possible to draw a connected graph on $\m$
in a flow-adapted way.

\begin{definition}
    Call a flow-adapted simple point process $\m$ \textdefn{connectible in a flow-adapted way}
    if there exists a connected 
    flow-adapted locally finite graph $\GG=\GG(\omega)$ with $\V(\GG)= \m = \sum_i
    \delta_{Y_i}$ almost surely and such that
    $1_{\{Y_i,Y_j\} \in \EE(\GG)}$ is measurable for all $i,j$.
\end{definition}

\begin{theorem}\label{locallyfinitegraphsufficient}
    Suppose $\X$ is unimodular and that $\m$ is a flow-adapted simple point process with intensity $\intensity \in (0,\infty)$
    that is connectible in a flow-adapted way.
    Then the Palm version of $\m$ is an $\X$-embedding of some
    unimodular network.
\end{theorem}

\begin{proof}
    Choose a random graph $\GG$ witnessing the fact that $\m$ is connectible.
    Consider some edge $\e \in \EE(\GG)$ from $X \in \m$ to $Y \in \m$.
    Without loss of generality, assume the mark space $\Xi = \X$.
    Let the mark of $(X,\e)$ be $X^{-1}Y$,
    and let $\beta(\omega)$ be the network with underlying graph $\GG(\omega)$
    and marks as just specified.
    Choose $\oo := e$ and let $[\GG,\oo]:= [\beta,e]$.
    Let $\psi$ a rooted automorphism of $(\GG,\v)$ be given, where $\v$ is any vertex of $\GG$.
    It will be shown that $\psi$ is the identity.
    Assume that $\psi$ fixes
    all vertices less than graph distance $k$ from $\v$.
    For each $Y\in \m$ of distance $k+1$ from $\v$,
    there is an $X \in \m$ of distance $k$ from $\v$ and an edge $\e$ from $X$ to
    $Y$. The mark of $(X,\e)$ is $X^{-1}Y$ and this must equal the mark of
    $\psi(X)=X$ between
    $X$ and $\psi(Y)$.
    But this mark is $X^{-1}\psi(Y)$.
    Thus $X^{-1}Y = X^{-1}\psi(Y)$ so that $Y=\psi(Y)$.
    By induction and using that $\GG$ is connected, one finds that $\psi$ is the
    identity automorphism.
    By construction $\V(\beta)=\m$, $\beta$ is flow-adapted,
    and $\omega \mapsto [\beta(\omega),e]$ is
    measurable on the set $\{e \in \m\}$.
    By \Cref{ppisunimodularnetwork}, $[\beta,e]$ is a unimodular
    network under $\P^\m$ on the set $\{e \in \m\}$.
    It will be shown that $\m$ is an $\X$-embedding of $[\GG,\oo]$.

    On the set $\{e \in \m\}$ one has that $\m$ can be reconstructed from $[\GG,\oo]$.
    The reconstruction procedure will be used to define an $\X$-embedding $\embedding$.
    Indeed, let $\t[\G,\o,\v] := \prod_{i=0}^{k-1} \xi_{\EE}(\v_i,\e_i)$
    where $\v_0\v_1\cdots\v_k$ is a path between the two arbitrary vertices
    $\o$ and $\v$ of $\G$, and $\e_i$ is
    the edge $\{\v_i,\v_{i+1}\}$, assuming the product is independent of the path
    chosen between.
    Let $\t[\G,\o,\v]:= e$ otherwise.
    The fact that  $\prod_{i=0}^{k-1} \xi_{\EE}(\v_i,\e_i)$ is required to be independent of path
    implies that for any $\v_1,\v_2,\v_3 \in \V(\G)$ one has
    $\t[\G,\v_1,\v_3] = \t[\G,\v_1,\v_2] \t[\G,\v_2,\v_3]$ and in particular
    that $\t[\G,\v,\o] = \t[\G,\o,\v]^{-1}$ for any $\o,\v \in \V(\G)$.

    For $[\G,\o] \in \calGs$ such that $(\G,\v)$ has no rooted automorphisms for any $\v \in \V(\G)$, define 
    \[
        \embedding[\G,\o] := \left\{ \t[\G,\o,\v] : \v \in \V(\G)  \right\},
    \]
    otherwise define $\embedding[\G,\o] := \{ e \}$.
    Also for $[\G,\o] \in \calGs$ such that $(\G,\v)$ has no rooted automorphisms for any $\v \in \V(\G)$,
    one has for each $\v \in \V(\G)$ that
    \begin{align*}
        \embedding[\G,\v]  
        &=  \left\{ \t[\G,\v,\v'] : \v' \in \V(\G)  \right\}\\
        &=  \left\{ \t[\G,\o,\v]^{-1} \t[\G,\o,\v] \t[\G,\v,\v'] : \v' \in \V(\G)  \right\}\\
        &=  \left\{ \t[\G,\o,\v]^{-1} \t[\G,\o,\v'] : \v' \in \V(\G)  \right\}\\
        &= T_{\t[\G,\o,\v]^{-1}} \embedding[\G,\o].
    \end{align*}
    If $[\G,\o]$ is such that some $v \in \V(\G)$ is such that $(\G,v)$ has a rooted automorphism,
    then then same is true of $[\G,\v]$, so that $\embedding[\G,\v] = \{e\} = T_{\t[\G,\o,\v]^{-1}}\embedding[\G,\o]$
    for the only vertex $\v=\o \in \V(\G)$.

    One may then recover $\m$ from $[\GG,\oo]$ on the set $\{e \in \m\}$ in the following way.
    Consider a path $\v_0\v_1\cdots \v_k$ in $[\GG,\oo]$ starting and ending at the root.
    Since there are no rooted automorphisms of $(\GG,\oo)$ on the set $\{e \in
    \m\}$ one may
    uniquely choose $X_0:=e,X_1,\ldots,X_{k-1},X_{k}:=e \in \X$ such that 
    $\xi_{\EE}(\v_i,\e_i) =  X_i^{-1}X_{i+1}$ for each $i$.
    Then $\prod_{i=0}^{k-1}\xi_{\EE}(\v_i,\e_i) = X_0^{-1} X_k=e$ regardless
    of the choice of $\v_0\v_1\cdots\v_k$ so long as the path starts and ends at
    the root.
    It follows that for an arbitrary path $\v_0\v_1\cdots \v_k$ one has that
    $\prod_{i=0}^{k-1}\xi_\EE(\v_i,\e_i)$ depends only on the
    endpoints  $\v_0$ and $\v_k$.
    Thus, for any $X \in \m = \V(\GG)$, consider a path starting at the root $\oo \in \GG$
    and ending at $\v:=X$.
    Then $\t[\GG,\oo,\v] = e^{-1} X = X$.
    Thus $\m = \{\t[\GG,\oo,\v] : \v \in V(\GG)\} = \embedding[\GG,\oo]$ almost surely on the set $\{e \in \m\}$.
    Hence $\embedding$ is a witness to the fact that
    $\m$ is an $\X$-embedding of $[\GG,\oo]= [\beta, e]$ under $\P^\m$.
\end{proof}

The problem of finding which point processes $\m$ have Palm versions that are $\X$-embeddings of some unimodular
network now reduces to finding which $\m$ admit a connected flow-adapted locally finite graph on
$\m$.
It is conjectured that the requirement that $\m$ be connectible in a flow-adapted way is automatic.

\begin{conjecture}
    Suppose $\X$ is unimodular. Then all flow-adapted simple point processes $\m$ on $\X$
    are connectible in a flow-adapted way.
\end{conjecture}

Finally, some special cases of the conjecture are known to hold.

\begin{theorem}
    Let $\m$ be a flow-adapted simple point process of intensity $\intensity \in
    (0,\infty)$.
    Then $\m$ is connectible in a flow-adapted way in any of the following situations:
    \begin{enumerate}[label=(\alph*)]
        \item $\X$ is compact,
        \item $\X=\R^d$,
        \item there exists a point-shift $\H$ such that $G^\H$ is
            connected.
    \end{enumerate}
\end{theorem}

\begin{proof}
    If $\X$ is compact then $\P$-a.s.\ $\m(\X)<\infty$ and the complete graph on
    $\m$ suffices.
    If $\X=\R^d$, then Theorem 5.4 in \cite{baccelli2016networks} shows that the
    Delaunay graph of $\m$ suffices.
    If there exists a point-shift $\H$ such that $G^\H$ is connected, then the
    graph $G^\H$ suffices.
\end{proof}

\newpage
\section*{Appendix: Palm Calculus}

In this appendix, fix a flow-adapted point process
$\m$ of intensity $\intensity \in (0,\infty)$.
The necessity of this appendix is mostly to prove \Cref{palmpalmostsureequivalence} and show how it may be used
to translate definitions under $\P$ and $\P^\m$, a technique that is used extensively is this research.

The connection between $\P$ and $\P^\m$ is given by the refined Campbell theorem, abbreviated to C-L-M-M
for Campbell, Little, Mecke, and Matthes.

\begin{theorem}[C-L-M-M]\label{CLMM} \cite{last2008modern}
    For all $f :\Omega \times \X \to \R_+$ measurable,
    \[
        \E\int_\X f(\theta_x^{-1},x)\,\m(dx) = \intensity \E^{\m} \int_\X f(\theta_e,x)\,\Haar(dx).
    \]
\end{theorem}

It is possible to recover $\P$, up to the set on which $\m$ is the zero measure,
via the following inversion formula.
The zero measure on $\X$ is denoted $\zero$.

\begin{theorem}[Inversion Formula]\label{inversionformula} \cite{last2008modern}
    There exists a bounded measurable $K:\Omega \times \X \to \R_+$ such that
    \begin{equation}\label{generalrandomunitintegral}
        \int_\X K(\theta_e,x)\,\m(dx) = 1_{\m \neq \zero},
    \end{equation}
    and for all $K :\Omega \times \X \to \R_+$ (not necessarily bounded) $\P$-a.s.\ satisfying
    \eqref{generalrandomunitintegral}, it holds that
    \begin{equation}\label{generalinversionequation}
        \E[1_{\m\neq \zero} f] = \intensity \E^\m\int_\X
        f(\theta_x)K(\theta_x,x)\,\Haar(dx)
    \end{equation}
    for all measurable $f:\Omega \to \R_+$.
\end{theorem}

\begin{proposition}\label{shiftinvariant}
    If $A \in \A$ is \textdefn{shift-invariant} in the sense that 
    $A = \theta_x^{-1} A$ for all
    $x\in\X$, then
    \[
        \P(A) = 1 \implies \P^\m(A) = 1 \implies \P(A \mid \m \neq \zero) = 1.
    \]
    In particular, if $\{\m = \zero\} \subseteq A$ then
    \[
        \P(A) = 1 \iff \P^\m(A) = 1.
    \]
\end{proposition}

\begin{proof}
    Suppose $\P(A)=1$.
    From the definition of Palm probabilities, for $B \in \B(\X)$ such that $\Haar(B) \in (0,\infty)$,
    \begin{align*}
        \P^\m(A) 
        &= \frac{1}{\intensity \Haar(B)}\E\int_\X 1_{x \in B}1_{\theta_x^{-1} \in A}\,\m(dx)\\
        &= \frac{1}{\intensity \Haar(B)}\E\int_\X 1_{x \in B}1_{A}\,\m(dx) & \text{(shift-invariance of $A$)}\\
        &= \frac{1}{\intensity \Haar(B)}\E [1_A \m(B)]\\
        &= \frac{1}{\intensity \Haar(B)}\E[ \m(B)] & (\P(A)=1)\\
        &= 1.
    \end{align*}

    Next suppose $\P^\m(A) = 1$.
    Then from \Cref{inversionformula} there is measurable $K:\Omega \times \X \to \R$ such that
    \begin{align*}
        \P(A \cap \{\m \neq \zero\}) 
        &= \E[1_{\m \neq \zero} 1_{A}] \\
        &= \intensity \E^\m\int_\X 1_{\theta_x \in A}
        K(\theta_x,x)\,\Haar(dx)
        && \text{(inversion formula)}\\
        &= \intensity \E^\m\left[1_A\int_\X K(\theta_x,x)\,\Haar(dx)\right]
        && \text{(shift-invariance of $A$)}\\
        &= \intensity \E^\m\left[\int_\X K(\theta_x,x)\,\Haar(dx)\right]
        && (\P^\m(A)=1)\\
        &= \E[1_{\m \neq \zero}\cdot 1]
        && \text{(inversion formula)}\\
        &= \P(\m \neq \zero).
    \end{align*}
    Dividing by $\P(\m \neq \zero) > 0$ gives $\P(A \mid \m \neq \zero) = 1$, and
    if $\{ \m = \zero \} \subseteq A$, then
    \[
        \P(A) = \P(A \cap \{\m \neq \zero\}) + \P(A \cap \{\m = \zero\}) =
        \P(\m \neq \zero) + \P(\m = \zero) = 1.
    \]
\end{proof}

\begin{lemma}\label{palmtop}
    Let $A \in \A$. 
    Then
    \[
        \P^\m(A) = 1 \iff \P(\m(x \in \X: \theta_x^{-1} \notin A) = 0) = 1.
    \]
\end{lemma}

\begin{proof}
    By replacing $A$ with its complement it is equivalent to show $\P^\m(A) = 0$ if and only if $\P(\m(x \in \X:
    \theta_x^{-1} \in A) > 0) = 0$.
    Note that it is the joint measurability of the action $(\omega,x)\mapsto \theta_x\omega$
    that lets one conclude for $B \in \B(\X)$ that sets like
    \[
        \{\m(x \in \X: x\in B, \theta_x^{-1} \in A) > 0\}
    \]
    are measurable.

    If $\P^\m(A) = 0$, then for $B \in \B(\X)$ such that
    $\Haar(B) \in(0,\infty)$,
    \begin{align*}
        0
        &=\P^\m(A) \\
        &= \frac{1}{\intensity \Haar(B)}\E\int_\X 1_{x \in B}1_{\theta_x^{-1}
            \in A}\,\m(dx)\\
        &= \frac{1}{\intensity \Haar(B)}\E[\m(x \in \X: x\in B, \theta_x^{-1} \in A)].
    \end{align*}
    Thus 
    $\E[\m(x \in \X: x\in B, \theta_x^{-1} \in A)] = 0$
    and taking relatively compact $B$ increasing to $\X$ one finds
    $\E[\m(x \in \X: \theta_x^{-1} \in A)] = 0$,
    so $\P(\m(x \in \X: \theta_x^{-1} \in A) > 0)=0$.

    Conversely, suppose 
    $\P(\m(x \in \X: \theta_x^{-1} \in A) > 0)=0$.
    Then for $B \in \B(\X)$ with $\Haar(B)\in (0,\infty)$,
    \begin{align*}
        \P^\m(A)
        &= \frac{1}{\intensity \Haar(B)}\E\int_\X 1_{x \in B}1_{\theta_x^{-1}
            \in A}\,\m(dx)\\
        &= \frac{1}{\intensity \Haar(B)}\E[\m(x \in \X: x\in B, \theta_x^{-1}
        \in A)]\\
        &\leq \frac{1}{\intensity \Haar(B)}\E[\m(x \in \X: \theta_x^{-1} \in A)]\\
        &= 0,
    \end{align*}
    completing the proof.
\end{proof}

It is now possible to prove~\Cref{palmpalmostsureequivalence}, which is restated here for clarity.

{
    \renewcommand{\thetheorem}{\ref{palmpalmostsureequivalence}}
    \begin{theorem}\label{palmpalmostsureequivalenceappendix}
        Let $A \in \A$. Then the following are equivalent:
        \begin{enumerate}[label=(\alph*)]
            \item $\P^\m(A) = 1$,
            \item $\P(\m(x\in \X: \theta_x^{-1} \notin A) =0)=1$,
            \item $\P^\m(\m(x\in \X: \theta_x^{-1} \notin A) =0)=1$.
        \end{enumerate}
    \end{theorem}
    \addtocounter{theorem}{-1}
}

\begin{proof}
    \begin{enumerate}
        \item[]
        \item[] $\emph{(a)}$ $\iff$ $\emph{(b)}$: 
        This is the content of \Cref{palmtop}.
        \item[] $\emph{(b)}\iff\emph{(c)}$: This follows 
            from \Cref{shiftinvariant} and the fact
            that the event $\{ \m(x\in \X: \theta_x^{-1} \notin A) =0 \}$ contains
            $\{\m = \zero\}$ and is shift-invariant.
            To wit, for all $y \in \X$,
            \begin{align*}
                &\theta_y^{-1}\omega \in \{ \m(x\in \X: \theta_x^{-1} \notin A) =0 \}\\
                \iff& \m(\theta_y^{-1}\omega,\{x \in
                \X:\theta_x^{-1}\theta_y^{-1}\omega \notin  A\})=0\\
                \iff& \m(\omega, \{yx: x \in \X, \theta_{yx}^{-1}\omega \notin
                A\})\\
                \iff& \m(\omega, \{x \in \X, \theta_{x}^{-1}\omega \notin A\})= 0\\
                \iff&\omega \in \{ \m(x \in \X, \theta_{x}^{-1}\omega \notin
                A)=0\}.
            \end{align*}
    \end{enumerate}
\end{proof}

\begin{example}\label{convertPandPalmDefn}
    Fix some measurable space $(S,\Sigma)$ and a measurable $f:\Omega \to S$.
    Define $F:\Omega\times \X \to S$ by 
    $F(\omega,x) := f(\theta_x^{-1}\omega)$
    for all $\omega\in \Omega, X \in \m(\omega)$, and $F(\omega,x)$ may be defined arbitrarily otherwise.
    It will be shown that knowing $F$ up to a $\P$- or $\P^\m$-null set on the support of $\m$ is equivalent to knowing $f$
    up to a $\P^\m$-null set.
    Indeed, suppose $f=f'$, $\P^\m$-a.s., then it will be shown that the corresponding $F,F'$
    agree $\P,\P^\m$-a.s.\ on the support of $\m$.
    By \Cref{palmpalmostsureequivalence}, $\P$- and $\P^\m$-a.e.\ $\omega \in \Omega$
    has for all $X \in \m(\omega)$ that
    $f(\theta_X^{-1}\omega) = f'(\theta_X^{-1}\omega)$, i.e. $F(\omega,X)=F'(\omega,X)$.
    Similarly, if either $\P$-a.e.\ or $\P^\m$-a.e.\ $\omega \in \Omega$ is such that 
    $F(\omega,X)=F'(\omega,X)$
    for all $X \in \m(\omega)$, then
    \begin{align*}
        f(\theta_X^{-1}\omega) 
        &= F(\omega,X) \\
        &= F'(\omega,X)\\
        &= f'(\theta_X^{-1}\omega),
    \end{align*}
    for $\P$-a.e.\ or $\P^\m$-a.e.\ $\omega \in \Omega, X \in \m(\omega)$, so
    by \Cref{palmpalmostsureequivalence}
    one finds that $f = f'$, $\P^\m$-a.s.
    Thus, $f$ may be defined under $\P^\m$ or $F$ may be defined under $\P$ or $\P^\m$, whichever is more convenient.
\end{example}

Finally, the following standard result is needed in \Cref{sec:classification-counterexample}.

\begin{theorem}[Slivnyak-Mecke Theorem]\label{slivnyakmecke} \cite{daley2007introduction}
    The distribution of $\m$ under $\P^\m$
    is the same as the distribution of $\m+\delta_e$ under $\P$ if and only if
    $\m$ is a homogeneous Poisson point process with intensity $\intensity$ under $\P$.
\end{theorem}

\begingroup
    \newpage
    \section*{References}
    \renewcommand{\section}[2]{}
    \bibliographystyle{abbrv}
    \bibliography{point-shiftfs-unimodular.bbl}
\endgroup

\end{document}